\documentclass[a4paper,12pt,reqno]{amsart}
\usepackage{a4wide}

\usepackage{fourier}
\usepackage{amsmath}
\usepackage{amsfonts}
\usepackage{amsthm}
\usepackage{cite}
\usepackage{hyperref}
\usepackage[toc,page]{appendix}
\usepackage{cancel}
\usepackage[normalem]{ulem}

\usepackage[usenames]{color}
\newcommand{\R}{\mathbb{R}}

\newtheorem{theorem}{Theorem}[section]
\newtheorem{lemma}[theorem]{Lemma}
\newtheorem{corollary}[theorem]{Corollary}
\newtheorem{prop}[theorem]{Proposition}

\newtheorem{remark}{Remark}
\newtheorem{example}{Example}

\newcommand{\veps}{\varepsilon}
\newcommand{\eps}{\epsilon}

\DeclareMathOperator{\Tr}{Tr}

\author[D. Ghilli]{Daria Ghilli$^*$}
\address{$^*$ Daria Ghilli: 
 Department of Economics and Management, University of Pavia, Via S. Felice al Monastero 5, 27100, Pavia, ITALIA}
\email{daria.ghilli@unipv.it}
\author[C. Marchi]{Claudio Marchi$^\S$}
\address{$^\S$ Claudio Marchi: Dipartimento di Ingegneria dell'Informazione \& Dipartimento di matematica ``Tullio Levi Civita'', 
Università degli studi di Padova, Via Gradenigo 6/B, 35131 Padova, ITALIA}
\email{marchi@math.unipd.it}
\title[Rate of convergence for singular perturbations of HJ equations in unbounded spaces]{Rate of convergence for singular perturbations of Hamilton-Jacobi equations in unbounded spaces}

\begin{document}

\date{\today}

\begin{abstract} 
We prove rate of convergence results for singular perturbations of Hamilton-Jacobi equations in unbounded spaces where the fast operator is linear, uniformly elliptic and has an Ornstein-Uhlenbeck-type drift. The slow operator is a fully nonlinear elliptic operator while the source term is assumed only locally H\"older continuous in both fast and slow variables. We obtain several rates of convergence according on the regularity of the source term.
\end{abstract}

\maketitle
\date{\today}

\bigskip

\bigskip 


\section{Introduction}
This paper is devoted to studying the limit behaviour as $\eps \to 0$ of nonlinear Hamilton-Jacobi equation of the form
\begin{equation}\label{eq:HJ-eps3}
u^\veps(x,y) + H\left(x,D_xu^\veps,D^2_{xx}u^\veps\right) +\frac{1}{\veps}\mathcal{L}\left(y, D_y u^\veps, D^2_{yy} u^\veps\right)+f(x,y)= 0 \qquad\textrm{in }\R^n\times\R^m
\end{equation}
where $H$ is a degenerate elliptic Hamiltonian, $f$ is a bounded locally H\"older continuous function and $\mathcal{L}$ is a linear operator
\begin{equation*}
\mathcal{L}\left(y, q, Y\right):=-\Tr\left(\tau(y)\tau(y)^TY\right)+b_0(y)\cdot q
\end{equation*}
(see below for more precise assumptions).
The study of the convergence as $\eps \to 0$ of the solutions $u^\eps$ to the equation \eqref{eq:HJ-eps3} is a singular perturbation problem in the whole space $\R^n \times \R^m$.

Singular perturbation problems for Hamilton-Jacobi equations have been thoroughly studied in the past years (see, for instance, \cite{AB3, N02, KP03} and references therein); the description of the whole literature goes beyond the purpose of the present paper. The PDE-approach to these problems consists in characterizing the value function~$u^\veps$ as a solution to a fully nonlinear PDE as in~\eqref{eq:HJ-eps3} and identifying its limit as $\veps \to0$ as the (unique) solution~$u$ of a limiting PDE. The theory of viscosity solution is the natural framework for this approach; we refer the reader to~\cite {BCD, CIL92} for an overview on this notion of solutions. In this framework, the ideas and methods for singular perturbations stem from the ones for periodic homogenisation problems started with the seminal papers~\cite{LPV, Evans,Ev2}.
According to this theory for singular perturbations, one expects that the solution to~\eqref{eq:HJ-eps3} converges locally uniformly to the solution~$u$ to the {\it effective} problem
\begin{equation}\label{eq:effprob}
u +H\left(x, D_xu, D^2_{xx}u\right)+\lambda(x)=0,
\end{equation}
where 
$\lambda(x)$ is the {\it ergodic constant} of the {\it cell problem}
\begin{equation}\label{7intro}
\mathcal{L}\left(y, D_yw, D^2_{yy}w\right)+f(x,y)=\lambda(x) \quad \mbox{ in } \R^m
\end{equation}
(incidentally, we mention that this equation also arises when the long time behaviour of the corresponding time dependent equation is examined).
It is also well known (see for instance, \cite{LPV}, \cite[Chapter 7]{AB3}, \cite{Ic11}) that existence and uniqueness of the solution $(w,\lambda)$ are issues to overcome. Hence, the first step is to solve problem~\eqref{7intro}, namely to prove that for each $x\in \R^n$ there exists a couple $(w,\lambda(x))$, with $w=w(y;x)$ such that problem~\eqref{7intro} is fulfilled, $\lambda(x)$ is uniquely determined and $w(\cdot;x)$ is uniquely determined up to an additive constant (in other words, there exists a unique $w$ such that $w(0;x)=0$).

It is well-known that the operator $\mathcal{L}$ is the infinitesimal generator of the following stochastic system
$$
dY_t=b_0(Y_t)dt+\sqrt{2}\tau(Y_t)dW_t,
$$
where $W_t$ is a standard $m$-dimensional Brownian motion.
Ergodicity of the process $Y_t$ is a key property that we need throughout the paper and which cannot be guaranteed by general diffusions $\tau$ and general drifts~$b_0$.
For this reason we focus our attention on the case where $\tau\tau^T$ is a positive definite matrix and the drift has the form
\begin{equation*}
b_0(y)=\alpha y + b(y),
\end{equation*}
where $\tau$ and $b$ are bounded, Lipschitz continuous, functions and $\alpha$ is a positive constant. Note that this situation encompasses the well-known Ornstein-Uhlenbeck process for the random motion of a particle under the influence of friction and several stochastic models for financial markets for which the periodic assumption is too restrictive (see \cite{FIL, BGL, FPS, BCG15, Ghilli18} and references therein). More specifically, under these assumptions, there exists a Lyapounov function for~$\mathcal{L}$, that is  a function $\chi$ (take for example $\chi(y)=|y|^2)$ such that
\begin{equation}\label{lyapounov}
\mathcal{L}\left(y, D_y\chi, D^2_{yy}\chi\right)\to + \infty \quad \mbox{ as }\, |y|\to \infty
\end{equation}
and a unique invariant measure $\mu$ for the process $Y_t$. Moreover, by these properties, the solution~$(w,\lambda)$ to~\eqref{7intro} is obtained as follows: $w$ is the limit as $\delta\to0$ of $u_\delta(\cdot)-u_\delta(0)$ where $u_\delta=u_\delta(y)$ is the unique bounded solution to
\begin{equation}\label{7intro_d}
\delta u_\delta +\mathcal{L}\left(y, D_yw, D^2_{yy}w\right)+f(x,y)=0 \quad \mbox{ in } \R^m
\end{equation}
while $\lambda$ is given by
\begin{equation}\label{lambda_intro}
\lambda(x)=\lim_{\delta\to 0} \delta u_\delta(0)=\int_{\R^m}f(x,y)\, d\mu(y)
\end{equation}
and $\mu$ is the invariant measure associated to process $Y_t$.

We recall that in~\cite{Ghilli18} the first author obtained the convergence for a class of singular perturbation problems in the whole space when the drift of the fast variables is of Ornstein-Uhlenbeck type. More precisely, she tackled the case of more general~$H$ but with constant matrix~$\tau$ and  Lipschitz continuous source term in the cell problem. This result was extended in~\cite{MMT18} to singular perturbations problems when the dynamics of the fast variables evolve in the whole space and ~$\mathcal{L}$ is a subelliptic operator and the source term is still Lipschitz continuous. We also mention that ergodic problems for viscous HJ equations with superlinear growth and inward-pointing drifts have been studied in~\cite{M14, CI19}. We refer the reader also to the paper~\cite{FIL} for the connection with the long time behaviour of solutions of the Cauchy problem for semilinear parabolic equations with the Ornstein-Uhlenbeck operator in the whole $\R^m$.

On the other hand, Capuzzo Dolcetta and Ishii~\cite{CDI01} provided the first result on the rate of convergence (namely, an estimate of $\|u^\veps-u\|_\infty$) for periodic homogenization of first order equations. Afterwards, their techniques were extended also to homogenization of second order equations, still under periodic assumption (see \cite{CM09, CM11, CCM11, Marchi14, KL16,RPT20}; see also \cite{CS10} for different techniques for the stationary ergodic case). However, up to our knowledge, no rate of convergence result for singular perturbation of second order HJ problems are available, even for the periodic case.

The main purpose of this paper is to obtain an estimate of the rate of convergence of~$u^\veps$ to~$u$, namely an estimate of $|u^\veps(x,y)-u(x)|$. As a byproduct we deduce that $u^\veps$ converges locally uniformly to $u$.
Several intermediate steps for achieving these two purposes are new and, in our opinion, they have an independent interest; in particular, we shall obtain the following results:  $(i)$ we establish the H\"older continuity of the solution to~\eqref{7intro_d} independently of $\delta $ and of the ellipticity of $\tau\tau^T$ (see Theorem~\ref{holder}), $(ii)$ we solve the cell problem~\eqref{7intro} on the whole space $\R^m$ with a source term~$f$ only locally H\"older continuous and with a noncostant matrix~$\tau$ (see Proposition~\ref{prop:nota}), $(iii)$ we establish a continuous dependence result in $x$ of the solution to~\eqref{7intro} with $w(0;x)=0$, namely and estimate for $|w(\cdot;x_1)-w(\cdot;x_2)|$ (see Proposition~\ref{prop:stimaccia}).

The main achievement of the present paper, stated in Theorem~\ref{thm:rategc}, is a rate of convergence results of $u_\eps$ to $u$ of the following type: under some compatibility condition between~$\alpha$ and the coefficients of~$\mathcal L$ (see assumption~$(H2)$ below), for every compact $\mathcal{K}$ of $\R^m$, there exists a constant $K$ (independent of $\eps$) such that for $\eps$ sufficiently small we have
$$
|u^\eps(x,y)-u(x)|\leq K\left(\eps\left|\log\eps\right|\right)^{\frac{\beta}{2}} \quad \forall (x,y) \in \R^n \times \mathcal{K},
$$
where  $\beta \in (0,1]$ is the H\"older exponent of the source~$f$.
 We remark that, for~$\tau$ and~$b$ constant, the compatibility assumption~$(H2)$ reduces to: $\alpha>0$; hence, our result applies to standard Ornstein-Uhlenbeck operators.
Moreover, if the source~$f$ has separated variables, namely $f(x,y)=h(x)g(y)$, we can drop assumption~$(H2)$ and we identify two subcases depending on the degree of regularity of the function $h$: in Theorem~\ref{thm:ratehlip}, for $h$ Lipschitz, we prove that for $\eps$ sufficiently small it holds
$$
|u^\eps(x,y)-u(x)|\leq K\left(\eps\left|\log \eps\right|\right)^{\frac{1}{2}} \quad \forall (x,y) \in \R^n \times \mathcal{K},
$$
while in Theorem~\ref{thm:ratehreg} for $h \in C^2(\R^n)$ we prove that for $\eps$ sufficiently small it holds
$$
|u^\eps(x,y)-u(x)|\leq K\eps\left|\log\eps\right| \quad \forall (x,y) \in \R^n \times \mathcal{K}.
$$

Our techniques will be in part inspired by the methods used for the rate of convergence in the periodic homogenization (see \cite{CM11, CM09, CCM11}) combined with the methods developed in \cite{FIL, Ghilli18} to solve singular perturbation problems in all the space. They strongly rely on the existence of a Lyapounov function, on a logarithmic growth of $w$ in $y$ and on the continuous dependence of $w$ in $x$. 

An illustrating case that we have in mind arises for stochastic optimal control problem: the operator~$H$ in \eqref{eq:HJ-eps3}  is of the following type
\begin{equation}\label{H_control}
H(x, p, X)=\min_{u \in U} \left\{-\phi(x,u)\cdot p-\mbox{trace}\left[\sigma(x,u)\sigma^T(x,u)X\right]\right\},
\end{equation}
where  $U$ is  a compact metric set while $\phi$ and $\sigma$ are bounded Lipschitz continuous functions. In this case the unique bounded solution to equation \eqref{eq:HJ-eps3} is characterized as the value function
$$
u^\veps(x,y)=\sup_{u \in \mathcal{U}}\mathbb{E}\left[\int_0^{+\infty} -e^{-t} f(X_t, Y_t)\,dt\right],
$$
where $\mathbb{E}$ denotes the expectation, $\mathcal{U}$ is the set of progressively measurable processes with values in $U$ and
\begin{equation}\label{esempio_dyn}
\left\{\begin{array}{lll}
(i)&\quad dX_t=\phi(X_t,u_t)dt+\sqrt{2}\sigma(X_t,u_t)dW_t,&\quad X_0=x \in \R^n,\\
(ii)&\quad dY_t=\frac{1}{\veps}\left[\alpha Y_t+b(Y_t)\right]dt+\sqrt{\frac{2}{\veps}}\tau(Y_t)dW_t, &\quad Y_0=y \in \R^m.
\end{array}\right.
\end{equation}
In this case the value function $u^\eps$ converges to the unique bounded solution~$u$ to the effective problem \eqref{eq:effprob} which in turns can be expressed as the value function
$$
u(x)=\sup_{u \in \mathcal{U}}\mathbb{E}\left[\int_0^{+\infty} -e^{-t} \lambda(X_t)\,dt\right]
$$
where the process $X_t$ still obeys to~\eqref{esempio_dyn}-(i). In this example, as $\veps\to 0$, the dynamics do not depend any more on the fast variable $Y$ and the cost $f$ is replaced by $\lambda$ which is its average with respect to the invariant measure~$\mu$, see~\eqref{lambda_intro}.

Another motivation for our study has been the papers \cite{BCG15, Ghilli18}, where singular perturbations problems 
where studied for their applications to large deviations, pricing of options near maturity and asymptotic formula for implied volatility. Nevertheless we remark that the equations studied in \cite{BCG15, Ghilli18} are more general than ours: 
 the study of more general singular perturbations problems (as the ones in \cite{BCG15, Ghilli18}) will be subject of future research.

The paper is structured as follows. In Section \ref{sec:pre} we give the main assumptions and set some notations. In Section~\ref{sec:holder}, for locally H\"older continuous cost~$f$, we prove the local H\"older regularity of the solution of the approximating ergodic problem~\eqref{7intro_d}, we solve the ergodic problem~\eqref{7intro} also obtaining a continuous dependence result for its solution w.r.t.~$x$ and, mainly, we establish the rate of convergence in Theorem \ref{thm:rategc}. In Section \ref{sec:partcases} we consider the particular cases in which the cost $f$ has separated variables and, depending on its regularity in $x$, we prove better rate of convergence (w.r.t. Theorem~\ref{thm:rategc}) in Theorem~\ref{thm:ratehlip} and in Theorem~\ref{thm:ratehreg}.

\section{Standing assumptions and notations}\label{sec:pre}
Throughout this paper, unless otherwise explicitly stated,  we shall assume the following
\begin{itemize}
\item[$(C)$] The comparison principle holds for \eqref{eq:HJ-eps3}.
\item[$(L)$] The operator $\mathcal{L}$ has the following form
\[
\mathcal{L}(y, q, Y):=-\mbox{tr}\left(\tau(y)\tau(y)^TY\right)+\alpha y\cdot q+b(y)\cdot q,
\]
where $\tau$ has bounded Lipschitz continuous coefficients with Lipschitz constant $L_\tau$, $b$ is bounded and Lipschitz continuous with Lipschitz constant $L_b$ and $\alpha>0$. Moreover we assume that $\tau \tau^T$ is uniformly non degenerate, that is, there exists $\theta >0$ such that
\begin{equation}\label{eq:tauund}
\xi\tau(y)\tau(y)^T \xi^T\geq \theta |\xi|^2, \quad \forall y \in \R^m, \xi \in \R^m.
\end{equation}	
\item[$(H)$] There exists $C>0$ such that, for all $x, p \in \R^n, X, Y \in \mathbb{S}^n$
$$H(x,0,0)\leq C,$$
$$
H\left(x, p, X\right)\leq H\left(x, p, Y\right) \mbox{ for } X\geq Y,
$$
$$
\left|H\left(x,p,X\right)-H\left(y,q,Y\right)\right|\leq C\left(|p-q|+\left|X-Y\right|\right)+C|x-y|\left(1+|p|+|X|\right).
$$
\item[$(F)$] The function $f$ is continuous and bounded: $\|f\|_\infty\leq C_1$.
\end{itemize}
The operator $\mathcal{L}$ is uniformly elliptic because of~\eqref{eq:tauund} while the operator~$H$ is only degenerate elliptic. We refer to~\cite{CIL92} for an overview for comparison principle for these operator.
\begin{example}
Consider the case where $H$ is the Hamiltonian associated to a stochastic optimal control problem as in~\eqref{H_control}. Assume that $\psi(x,u)$ is bounded and Lipschitz continuous in~$x$ uniformly in~$u$ where $\psi=\phi,\sigma$. Then assumptions~$(C)$ and~$(H)$ are verified; see~\cite{CIL92}.
\end{example}
\begin{lemma}\label{lemma:Exis_udelta}
There exists a unique bounded viscosity solution $u^\veps$ to problem~\eqref{eq:HJ-eps3}. Moreover there holds $\|u^\veps\|_\infty\leq C+C_1$.
\end{lemma}
\begin{proof} We observe that the functions $u^\pm=\pm(C+C_1)$ are respectively a super- and a subsolution to problem~\eqref{eq:HJ-eps3}. Hence, a standard application of Perron's method (see for instance \cite{CIL92}) ensures the existence of a viscosity solution $u^\veps$ to problem~\eqref{eq:HJ-eps3} with $u^-\leq u^\veps\leq u^+$. On the other hand, the comparison principle guarantees the uniqueness of the solution.
\end{proof}

%

\section{H\"older continuous source term}\label{sec:holder}
In this section we assume the following hypotheses:
\begin{itemize}
\item[$(H1)$] the function $f$ is locally H\"older continuous in $y$ (uniformly in $x$): there exist $\gamma\in(0,1]$ and $C_2\in(0,\infty)$ such that
\[
|f(x,y_1)-f(x,y_2)|\leq C_2|y_1-y_2|^\gamma\left[\log\left(1+|y_1|^2\right)+\log\left(1+|y_2|^2\right)+1\right]\qquad\forall x\in\R^n,\, y_1,y_2\in\R^m;
\]
\item[$(H2)$] the parameter $\alpha$ satisfies $\alpha>L_b+L_\tau^2(m+2-\gamma)$, where $L_b$ and $L_\tau$ are defined in $(L)$;
\item[$(H3)$] There exist two positive constants  $C_3$ and $\beta\in(0,1]$ such that: for any $x,\bar x\in\R^n$, the function $F(\cdot):=f(x,\cdot)-f(\bar x,\cdot)$ fulfills
\[
|F(y_1)-F(y_2)|\leq |y_1-y_2|^\gamma C_3|x-\bar x|^\beta\left[\log\left(1+|y_1|^2\right)+\log\left(1+|y_2|^2\right)+1\right]\qquad\forall y_1,y_2\in\R^m;
\]
\item[$(H4)$]  There exist two positive constants  $C_3$ and $\beta\in(0,1]$ such that: for any $x,\bar x\in\R^n$, the function $F(\cdot):=f(x,\cdot)-f(\bar x,\cdot)$ fulfills
\[
\|F\|_\infty\leq C_3|x-\bar x|^\beta.
\]
\end{itemize}

\begin{remark}
We remark that assumption~$(H2)$ covers the case of the standard Ornstein-Uhlenbeck operator, that is when $\tau$ and $b$ are constant; 
actually in this case $(H2)$ reduces to $\alpha >0$. 
However, in Section \ref{sec:partcases} we drop this assumption for the case of $f(x,y)=h(x)g(y)$ Lipschitz continuous.
\end{remark}
This section is organized as follows: in subsection \ref{subsec:cellp1} we analyse the cell problem while in subsection~\ref{subsec:effect} we establish some properties of the effective problem. 
Finally subsection~\ref{subsec:rateconv} contains our main result on the rate of convergence, Theorem \ref{thm:rategc}.

\subsection{The cell problem}\label{subsec:cellp1}

For each $x\in\R^n$, we consider the cell problem
\begin{equation}\label{eq:cellpb3}
\mathcal{L}\left(y, D_yw, D^2_{yy}w\right)+f(x,y)=\lambda(x), \quad w(0)=0 \quad \mbox{ in } \R^m.
\end{equation}
In order to study this problem, it is expedient to introduce the approximating cell problem
\begin{equation}\label{delta_ergo}
\delta u_\delta(y)+\mathcal{L}\left(y, D_y u_\delta, D^2_{yy}u_\delta\right)=F(y)\qquad\textrm{in }\R^m,
\end{equation}
where, by assumptions $(H1)$ and $(F)$, the source~$F$ satisfies
\begin{itemize}
\item[$(F1)$] $||F(\cdot)||_\infty\leq K_F$, 
\item[$(F2)$] $|F(y_1)-F(y_2)|\leq C_F|y_1-y_2|^\gamma\left[\log\left(1+|y_1|^2\right)+\log\left(1+|y_2|^2\right)+1\right]$, $\forall y_1,y_2\in\R^m$,
\end{itemize}
for some constants $K_F, C_F>0, \gamma \in (0,1]$.

In the following theorem we prove the well-posedness of problem~\eqref{delta_ergo} and, mainly, a regularity estimate for its solution. This regularity property will be crucial in the proof of three results: $(i)$ the cell problem~\eqref{eq:cellpb3} has a solution~$w$ (see Proposition~\ref{prop:nota}), $(ii)$ $w$ has a logarithmic growth at infinity (see Lemma~\ref{estchiholder}), $(iii)$ a continuous dependence estimate of~$w$ with respect to~$x$. The proof is postponed to Appendix A.

\begin{theorem}\label{holder}
Assume $(F1), (F2)$ and $(H2)$.
Then:
\begin{itemize}
\item[$i)$] there exists a unique bounded (viscosity) solution~$u_\delta$ to problem~\eqref{delta_ergo}; moreover there holds: $\delta\|u_\delta\|_\infty\leq K_F$ for any $\delta\in(0,1]$.
\item[$ii)$] the unique solution $u_\delta$ found in point~$i)$ satisfies
\begin{equation*}
|u_\delta(y_1)-u_\delta(y_2)|\leq K_1|y_1-y_2|^\gamma\left[\log\left(1+|y_1|^2\right)+\log\left(1+|y_2|^2\right)+K_2\right],\qquad \forall y_1, y_2 \in \R^m,
\end{equation*}
for some positive constants $K_1$ and $K_2$ which depend only on $m, C_F, \gamma, \alpha, \tau$ and $b$ (and are independent of $\delta$) and $K_1$ has a linear dependence in~$C_F$.
\end{itemize}
\end{theorem}

\begin{remark}
It is worth to observe that the uniform ellipticity~\eqref{eq:tauund} of~$\mathcal L$ is not needed in the proof of Theorem~\ref{holder}.
\end{remark}
We solve the cell problem in the following proposition. 

\begin{prop}\label{prop:nota}
Assume $(H1)$ and $(H2)$. For any $x \in \R^n$ fixed,  there exists  a unique ergodic constant $\lambda(x)$ such that the  cell problem~\eqref{eq:cellpb3} admits a unique corrector $w=w(y;x) \in C^2(\R^m)$. Moreover the ergodic constant is given by
\begin{equation}\label{frm_lambda}
\lambda(x)=\int_{\R^m} f(x,y)d\mu(y) \qquad\forall x\in \R^n,
\end{equation}
where $\mu$ is the unique invariant measure associated to $\mathcal{L}$.
\end{prop}
\begin{proof}
This proof relies on the same arguments of~\cite[Proposition $4.2$ and Proposition $4.4$]{Ghilli18} (see also~\cite[Theorem 3.9]{MMT18} for similar arguments); indeed, in~\cite{Ghilli18} assumption~\cite[$(U)$]{Ghilli18} is not needed in the resolution of the ergodic problem. For the sake of completeness, we shall only give the main ingredients.\\
Fix $x\in\R^n$ and set $F(\cdot)=f(x,\cdot)$. Let $u_\delta$ be the solution to~\eqref{delta_ergo}; the function $v_\delta(\cdot):=u_\delta(\cdot)-u_\delta(0)$ is a solution to
$$
\delta v_\delta(y)+\mathcal{L}\left(y, D_y v_\delta, D^2_{yy}v_\delta\right)+\delta u_\delta(0)=F(y)\qquad\textrm{in }\R^m.
$$
Since $v_\delta(0)=0$, Theorem~\ref{holder}-(ii) yields that the family $\{v_\delta\}_\delta$ is uniformly locally bounded and uniformly locally H\"older continuous. By Ascoli-Arzela theorem (eventually passing to a subsequence that we still denote $v_\delta$), as $\delta\to0^+$, $v_\delta(\cdot)$ converges locally uniformly to some function~$w(\cdot;x)$ while $\delta v_\delta(\cdot)$ converges to zero  and $\delta u_\delta(0)$ converges to some constant~$\lambda=\lambda(x)$. By stability results, the couple~$(w(\cdot;x),\lambda(x))$ is a viscosity solution to~\eqref{eq:cellpb3}. By \cite[Theorem $1$ and Theorem $2$]{H}, $w$ is also a distributional solution to \eqref{eq:cellpb3}. Then the $C^2(\R^m)$-regularity of~$w(\cdot;x)$ stems from the uniform ellipticity~\eqref{eq:tauund} of~$\mathcal{L}$ and the H\"older regularity of~$f$ in~$(H1)$ (e.g. see \cite[Theorem $6.14$]{GT}).\\
The rest of the proof follows the same arguments of the aforementioned papers and we shall skip it.
\end{proof}
We can now deduce some properties of the function~$\lambda(x)$.
\begin{corollary}\label{corollary:A}
Besides the assumptions of Proposition~\ref{prop:nota}, assume also $(H4)$. Then, the function~$\lambda$ satisfies
$$
\|\lambda\|_\infty\leq C_1,\qquad |\lambda(x)-\lambda(\bar x)|\leq C_3|x-\bar x|^\beta \qquad \forall x,\bar x\in \R^n
$$
where $C_1$,~$C_3$ and~$\beta$ are the constants introduced respectively in~$(F)$ and in~$(H4)$.
\end{corollary}
\begin{proof}
The former inequality is an immediate consequence of Theorem~\ref{holder}-(i). In order to establish the latter estimate, it suffices to use assumptions~$(F)$ and~$(H4)$ and equation~\eqref{frm_lambda}.
\end{proof}
\begin{remark}
For more general situations, let us provide an alternative proof of Corollary~\ref{corollary:A} avoiding formula~\eqref{frm_lambda}. Let $u_\delta$ and~$\bar u_\delta$ solve~\eqref{delta_ergo} with~$F(\cdot)$ replaced with~$f(x,\cdot)$ and respectively with~$f(\bar x, \cdot)$. The functions $\overline u_\delta\pm C_3|x-\bar x|^\beta/\delta$ are respectively a super- and a subsolution to~\eqref{delta_ergo}. The comparison principle entails: $\delta\|u_\delta-\bar u_\delta\|_\infty\leq C_3|x-\bar x|^\beta$. Letting $\delta\to0$, we get the statement.
\end{remark}

Mainly by relying on the results of Theorem \ref{holder}, we prove the following estimates on the solution of the cell problem \eqref{eq:cellpb3} that will be crucial in the proof of the rate of convergence, stated in~Theorem \ref{thm:rategc}.

\begin{lemma}\label{estchiholder}
Assume $(H1)$ and $(H2)$. Let $w$ and~$\chi$ be the solution to the cell problem \eqref{eq:cellpb3} and respectively the Lyapounov function of \eqref{lyapounov}. Then, there exists a constant~$C_5>0$ such that
\begin{itemize}
\item[i)] $|w(y;x)|\leq C_5\left(1+\log(1+|y|^2)\right),\qquad\forall(x,y)\in\R^n\times\R^m$;
\item[ii)] $|w(y;x)|-\eta \chi(y)\leq  C_5\log C_5- C_5\log\eta +\eta,\qquad\forall(x,y)\in\R^n\times\R^m$.
\end{itemize}
\end{lemma}
\begin{proof}
(i). We note that the function $g(y)= C_5(1+\log(1+|y|^2))$ is a supersolution to~\eqref{delta_ergo} at least for $|y|$ sufficiently large (independently of $\delta$) provided that $C_5$ is sufficiently large (independently of $\delta$). The rest of the proof follows the same arguments of~\cite[Proposition $3.3$]{Ghilli18} replacing~\cite[Lemma $3.5$]{Ghilli18} with Theorem~\ref{holder} (see also \cite[Lemma 3.8]{MMT18} for similar arguments).

(ii). By point (i) and since $\chi(y)= |y|^2$, we immediately get that 
\begin{equation*}
|w(y;x)|-\eta \chi(y)\leq C_5\left(1+\log(y^2+1)\right)-\eta |y|^2.
\end{equation*}
Since the maximum of the right-hand side is attained for $y^2=\frac{C_5-\eta}{\eta}$, we get the statement.
\end{proof}

As a consequence of Theorem \ref{holder}, we are able to prove the following  results, that we need in the proof of the rate of convergence stated in Theorem \ref{thm:rategc}. Note that point v) establishes  a continuous dependence estimate  of the ergodic corrector.
\begin{prop}\label{prop:stimaccia}
Assume $(H1)$, $(H2)$, $(H3)$ and $(H4)$. For any $x$ and $\bar x$ fixed in $\R^n$, set $F(\cdot):=f(x,\cdot)-f(\bar x,\cdot)$. The problem
\begin{equation}\label{stimaccia}
\delta 	U_\delta(y) +\mathcal{L}\left(y, D_y U_\delta, D^2_{yy} U_\delta\right)+F(y)=0, \qquad \mbox{ in }\R^m
\end{equation}
admits exactly one bounded solution. Moreover, the following properties hold true:
\begin{itemize}
\item[i)] as $\delta\to0^+$, $\left\{\delta U_\delta (0)\right\}_{\delta}$ converges to $\lambda(x)-\lambda(\bar x)$;
\item[ii)] For $W_\delta(\cdot):=U_\delta(\cdot)-U_\delta (0)$, as $\delta\to0^+$, the sequence $\{W_\delta\}_{\delta}$ converges locally uniformly to $w(\cdot; x)-w(\cdot; \bar x)$, where $w(\cdot;x)$ and $w(\cdot;\bar x)$ are respectively the solution to problem~\eqref{eq:cellpb3} and the solution to the same problem with $x$ replaced by $\bar x$;
\item[iii)] for some constant $C_4$ independent of $x$, $\bar x$ and $\delta$, there holds 
\[
\left|U_\delta(y_1)-U_\delta(y_2)\right|\leq |y_1-y_2|^\gamma C_4|x-\bar x|^\beta\left[\log\left(1+|y_1|^2\right)+\log\left(1+|y_2|^2\right)+C_4\right]\qquad\forall y_1,y_2\in\R^m;
\]
\item[iv)] for some constant $C_4$ independent of $x$ and $\bar x$, the function $W(\cdot):=w(\cdot;x)-w(\cdot;\bar x)$ fulfills
\[
\left|W(y_1)-W(y_2)\right|\leq |y_1-y_2|^\gamma C_4|x-\bar x|^\beta\left[\log\left(1+|y_1|^2\right)+\log\left(1+|y_2|^2\right)+C_4\right]\qquad\forall y_1,y_2\in\R^m;
\]
\item[v)] the function $W(\cdot)$ introduced in point (iv) also fulfills
\[
\left|W(y)\right|\leq C_6|x-\bar x|^\beta[1+\log\left(1+|y|^2\right)]\qquad \forall y \in \R^m,
\]
for some constant $C_6$ independent of $x$ and $\bar x$.
\end{itemize}
\end{prop}
\begin{proof}
We observe that the linearity of equation~\eqref{stimaccia} entails that $U_\delta(\cdot)=u_\delta(\cdot;x)-u_\delta(\cdot;\bar x)$ where $u_\delta(\cdot;z)$ solves
$\delta u_\delta +\mathcal{L}\left(y, D_y u_\delta, D_{yy}^2 u_\delta\right)+f(y;z)=0$. By standard theory on ergodic problem, we deduce points i) and ii).\\
Point~iii) is an easy consequence of Theorem~\ref{holder} ii) with $K_F=C_F=C_3|x-\bar x|^\beta$.\\
Point~iv) follows from points~ii) and~iii).\\
Point~v) is obtained following the same arguments as those in the proof of Lemma~\ref{estchiholder} i), replacing the result in Theorem~\ref{holder} ii) with point~iii). Nevertheless, let us give some details in order to provide explicitly the dependence on $|x-\bar x|$. We observe that, by assumption~$(H4)$, we have: $\|F\|_\infty\leq C_3|x-\bar x|^\beta$. Hence, there exist two constants $C$ and $R$ (independent of $x$ and $\bar x$) such that the function $g(y):= C C_3|x-\bar x|^\beta\left[1+\log\left(1+|y|^2\right)\right]$ is a supersolution to problem~\eqref{delta_ergo} in $\R^m\setminus B(0,R)$. 
We observe that, by point~ iii), there holds
\[
\max_{\overline {B(0,R)}}U_\delta-U_\delta(0)\leq C_4|x-\bar x|^\beta  R^\gamma\left[2\log\left(1+R^2\right)+C_4\right]
\]
and we conclude arguing as in~\cite[Proposition $3.3$]{Ghilli18} (see also \cite[Lemma 3.8]{MMT18} for a similar argument).
\end{proof}

\subsection{Effective problem}\label{subsec:effect}
By virtue of the properties of~$\lambda(x)$ found in subsection~\ref{subsec:cellp1}, we can now tackle the effective problem~\eqref{eq:effprob}.
\begin{lemma}\label{lemma:ex_u}
Assume $(H1)$, $(H2)$ and $(H4)$.
\begin{itemize}
\item[i)] the comparison principle holds for problem~\eqref{eq:effprob};
\item [ii)] there exists a unique bounded solution~$u$ to problem~\eqref{eq:effprob}. Moreover, there holds: $\|u\|_\infty\leq C+C_1$ where~$C$ and~$C_1$ are the constants introduced respectively in~$(H)$ and in~$(F)$.
\end{itemize}
\end{lemma}
\begin{proof}
Assumption~$(C)$ and Corollary~\ref{corollary:A} entail that the comparison principle applies to the effective problem~\eqref{eq:effprob}. Following the same arguments as those in the proof of Lemma~\ref{lemma:Exis_udelta}, we achieve the second part of the statement.
\end{proof}

\subsection{Rate of convergence}\label{subsec:rateconv}
In the following theorem we prove our main result.
\begin{theorem}\label{thm:rategc}
Assume $(H1)$, $(H2)$, $(H3)$, $(H4)$. Let $u^\eps$ and~$u$ be respectively the solution to problem~\eqref{eq:HJ-eps3} and to the effective problem \eqref{eq:effprob}. For $\eps$ sufficiently small and for every compact $\mathcal K \subset \R^m$ there exists a constant $K\geq 0$ such that 
$$
\left|u^\eps(x,y)-u(x)\right|\le K
\left(\eps|\log\eps|\right)^{\frac{\beta}{2}}
\qquad \forall (x,y)\in \R^n\times \mathcal{K}
$$
where $\beta$ is the constant introduced in assumptions~$(H3)-(H4)$.
\end{theorem}
In order to prove Theorem \ref{thm:rategc}, for~$\rho\in(0,1)$, it is expedient to introduce the operators
\begin{equation}\label{eq:effbarrhoH}
\overline H_\rho(x,p,X)=\min_{\xi \in B_\rho(x)}\left[H(\xi, p, X)+\lambda(\xi)\right].
\end{equation}
The use of the operators $\overline H_\rho$ is in the spirit of the approximated Hamiltonians introduced in \cite{BJ02} and in the \textit{shaking of coefficients} method by Krylov (see \cite{Kr00}) and it was already applied in~\cite{CM11} in order to overcome the lack of uniform continuity of the Hamiltonian.
We consider the {\it approximated} effective problem
\begin{equation}\label{eq:effprobrho}
u_\rho(x)+\overline H_\rho\left(x, D_x u_\rho, D^2_{xx}u_\rho\right)=0;
\end{equation}
in the following Lemma, we gather some useful properties of this problem.
\begin{lemma}\label{lem:approxrho}
Problem~\eqref{eq:effprobrho} admits exactly one bounded solution~$u_\rho$. Moreover, the sequence $\left\{u_\rho \right\}_\rho$ converges locally uniformly to the solution $u$ to \eqref{eq:effprob} as $\rho \to 0$.
\end{lemma}
\begin{proof}
We observe that assumptions~$(C)$ and~$(L)$ and Corollary~\ref{corollary:A} ensure $|\overline H_\rho(x,0,0)|\leq C+C_1$ and that the comparison principle holds for problem~\eqref{eq:effprobrho}.
By the same arguments of the proof of Lemma~\ref{lemma:Exis_udelta}, there exists a unique bounded solution~$u_\rho$ to problem~\eqref{eq:effprobrho} which moreover fulfills: $\|u_\rho\|_\infty\leq C+C_1$ for any $\rho\in(0,1)$.

Moreover, since the sequence $\left\{u_\rho\right\}_\rho$ is equibounded, invoking \cite[Theorem V.1.7 or Corollary V.1.8]{BCD} we get that $u_\rho$ converges locally uniformly to~$u$ as $\rho\to0$.
\end{proof}
Before proving Theorem~\ref{thm:rategc}, we state the following proposition whose proof is postponed at the end of this section.

\begin{prop}\label{prop:ratehgc}
Assume $(H1)$, $(H2)$, $(H3)$, $(H4)$.  Let $u^\eps$ be the solution to the equation \eqref{eq:HJ-eps3} and $u_\rho$ the  solution to problem \eqref{eq:effprobrho}. For $\eps$ sufficiently small and for every compact $\mathcal{K} \subset \R^m$, there exists $K\geq 0$ (independent of $\rho \in (0,1)$) such that:
$$
u^\eps(x,y)-u_\rho(x) \le K
\left(\eps|\log\eps|\right)^{\frac{\beta}{2}}
\qquad \forall (x,y)\in \R^n\times \mathcal{K},
$$
where $\beta$ is the constant introduced in assumptions $(H3)-(H4)$.
\end{prop}
Now we prove Theorem \ref{thm:rategc}.
\begin{proof}[Proof of Theorem \ref{thm:rategc}]
By Lemma \ref{lem:approxrho} and Proposition \ref{prop:ratehgc}, as $\rho \to 0^+$, we obtain that, for every compact $\mathcal{K} \subset \R^m$, there exists $K\geq 0$ (independent of $\rho \in [0,1)$) such that:
$$
u^\eps(x,y)-u(x)\le K
\left(\eps|\log\eps|\right)^{\frac{\beta}{2}}\qquad \forall (x,y)\in \R^n\times \mathcal{K}.
$$
The other inequality of the statement is established in a similar manner so we omit the details of its proof and, for the sake of completeness, we just describe the main steps. For
\begin{equation*}
\overline H^\rho(x,p,X)=\max_{\xi \in B_\rho(x)}\left[H(\xi, p, X)+\lambda(\xi)\right],
\end{equation*}
consider the problem
\begin{equation}\label{16bis}
u^\rho(x)+\overline H^\rho\left(x, D_x u^\rho, D^2_{xx}u^\rho\right)=0.
\end{equation}
By the same arguments as those in the proof of Lemma~\ref{lem:approxrho}, we obtain that problem~\eqref{16bis} has a unique bounded solution~$u^\rho$ and that $\{u^\rho\}_\rho$ converges locally uniformly to~$u$ as $\rho\to0^+$.
Arguing as in the proof of Proposition \ref{prop:ratehgc}, we obtain that
for every compact $\mathcal{K} \subset \R^m$, there exists $K\geq 0$ (independent of $\rho \in (0,1)$) such that:
$$
u^\eps(x,y)-u^\rho(x)\ge K
\left(\eps|\log\eps|\right)^{\frac{\beta}{2}}\qquad \forall (x,y)\in \R^n\times \mathcal{K}.
$$
Indeed, we observe that the function $\psi(x,y)=u^\eps(x,y)-u_\rho(x)-\eps w(y;x)+\eps \eta \chi(y)+\sigma |x|^2$ attains its global minimum at some point~$(\bar x, \bar y)$ and the analogous estimates to those of Lemma \ref{lem:est} apply. Moreover we can apply  ii) of Lemma \ref{estchiholder} to estimate from below the function  $-w(y;x)+\eta \chi(y)$. 
The rest of the computations follows similarly as above, using that $u^\rho$ is a subsolution of equation \eqref{16bis} and that $u^\eps$ is a supersolution of \eqref{eq:HJ-eps3}.
\end{proof}
Next we prove Proposition \ref{prop:ratehgc}.
\begin{proof}[Proof of Proposition \ref{prop:ratehgc}]
We consider the function 
\begin{equation}\label{eq:psiprop11}
\psi(x,y)=u^\eps(x,y)-u_\rho(x)-\eps w(y;x)-\eps\eta\chi(y)-\sigma |x|^2,
\end{equation}
where $\rho, \eps, \eta, \sigma >0$ are sufficiently small and will be chosen later on. The function $\psi$ attains its global maximum at some point $(\bar x, \bar y)$. 
We introduce a new function
\begin{equation}\label{eq:tildepsiprop11}
\tilde{\psi}(x,y):=u^\eps(x,y)-u_\rho(x)-\eps w(y;\bar x)-\eps \eta \chi(y)-\sigma|x|^2-c|x-\bar x|^2
\end{equation}
for some $c>0$ to be fixed later on. We observe that
\begin{equation}\label{eq:comppsi2}
\tilde{\psi}(\bar x,  y)=\psi(\bar x,  y).
\end{equation}
Consider $\overline{B(\bar x, r)\times B\left(0,\frac{\tilde{C}}{\sqrt{\eps\eta}}\right)}:=\mathcal{B}\subset \R^n\times \R^m,$  where $\tilde C>0, r\in (0,1)$ will be chosen later on. Let us emphasize that $\tilde C$ will be chosen independent of $\eps$, $\eta$, $\rho$ and~$\sigma$. We observe that for each $(x,y) \in \partial \mathcal{B}$ we have the following cases.\\
\textbf{Case a:} \quad $x \in \partial B(\bar x, r)$. We have
\begin{equation}\label{eq:tildepsipsi}
\tilde\psi(x,y)\leq \psi(x,y) \quad \forall (x,y)  \in \partial B(\bar x, r)\times B\left(0,\frac{\tilde C}{\sqrt{\eps\eta}}\right)
\end{equation}
if and only if 
\begin{equation}\label{20}
cr^2\geq \eps [w(y;x)-w(y;\bar x)], \quad \forall y \in B\left(0,\frac{\tilde C}{\sqrt{\eps\eta}}\right).
\end{equation}
We observe that Proposition~\ref{prop:stimaccia}-(v) ensures
\begin{eqnarray*}
|w(y;x)-w(y;\bar x)|\leq C_6 r^\beta\left(1+\log\left(1+|y|^2\right)\right)\leq  \bar Cr^\beta\left(1+ \left|\log(\eps \eta)\right|\right)
\quad \forall y \in B\left(0,\frac{\tilde C}{\sqrt{\eps\eta}}\right)
\end{eqnarray*}
for some $\bar C>0$ independent of~$\eps$, $\eta$, $\rho$ and $\sigma$ (because $\tilde C$ will fulfill such independence). Hence, in order to have~\eqref{20}, it suffices to have
\begin{equation}\label{eqn:cr_holder}
cr^{2-\beta}\geq \bar C\eps\left(1+ \left|\log(\eps \eta)\right|\right).
\end{equation}
We remark that our choice for $r$ and $c$ will fulfill this inequality. Therefore if the previous condition holds and using \eqref{eq:comppsi2}, we have
\begin{equation}\label{eq:tildepsipsicasea}
\tilde \psi(x,y)\leq \psi(x,y)\leq \psi(\bar x, \bar y)=\tilde \psi(\bar x, \bar y) \quad \forall (x,y)  \in \partial B(\bar x, r)\times B\left(0,\frac{\tilde C}{\sqrt{\eps\eta}}\right).
\end{equation}
\textbf{Case b:} \quad $y  \in \partial B\left(0,\frac{\tilde C}{\sqrt{\eps\eta}}\right)$. In order to have
\begin{equation}\label{eq:tildepsipsizero}
\tilde{\psi}(x,y)\leq \psi(\bar x,0), \quad \forall (x,y)\in B(\bar x, r)\times \partial B\left(0,\frac{\tilde C}{\sqrt{\eps\eta}}\right)
\end{equation}
it suffices 
$\eps \eta \chi(y)\geq 4M+\eps |w(y; \bar x)|-\sigma\left(|x|^2-|\bar x|^2\right),$
where $M=\max\left\{||u^\eps||_\infty, ||u_\rho||_\infty\right\}$.
By Lemma \ref{estchiholder} i) and recalling that $\chi(y)=\tilde C^2(\eps\eta)^{-1}$ for $y  \in \partial B\left(0,\frac{\tilde C}{\sqrt{\eps\eta}}\right)$, it suffices to have
$$
\tilde C^2\geq 4M+\eps C_5\left(1+\log\left(1+\frac{\tilde C^2}{\eps \eta}\right)\right)+\sigma\left(|\bar x|^2-|x|^2\right).
$$
Note that
$$
\sigma\left(|\bar x|^2-| x|^2\right)= \sigma(|\bar x|-|x|)(|\bar x|+|x|)
\leq \sigma r(2|\bar x|+r);
$$
hence, since $r \in (0,1)$, it suffices to have
$$
\tilde C^2\geq 4M+\eps C_5\left(1+\log\left(1+\frac{\tilde C^2}{\eps \eta}\right)\right)+\sigma (2|\bar x|+1),
$$
and recalling that by Lemma \ref{lem:est} i) (with $v=u_\rho$), for $\eps$ sufficiently small and as $\sigma \to 0$, we have
\begin{equation}\label{limsigmabis}
\lim_{\sigma \to 0}\sigma |\bar x|=0,
\end{equation}
we deduce that the previous inequality always holds for $\tilde C$ large enough when $\sigma $ and $\eps$ tend to zero (and, as we  will choose, $\eta=\eps$). (Note that, for $\eta=\eps$ sufficiently small and $\sigma$ sufficiently small, we can choose $\tilde C=\sqrt{5M}$).
Therefore by \eqref{eq:comppsi2} and if the previous condition holds, we have
\begin{equation}\label{eq:tildepsipsicaseb}
\tilde \psi(x,y)\leq \psi(\bar x, 0)\leq \psi(\bar x, \bar y)=\tilde \psi(\bar x, \bar y)\quad \forall (x,y)  \in  B(\bar x, r)\times \partial B\left(0,\frac{\tilde C}{\sqrt{\eps\eta}}\right).
\end{equation}
In conclusion, we have
\begin{equation}\label{eq:tildepsipsitotal}
\tilde{\psi}(x,y)\leq \tilde{\psi}(\bar x, \bar y), \quad \forall (x,y) \in \partial \mathcal{B}.
\end{equation}
For $\eps$ and $\sigma$ small enough, by Lemma \ref{lem:est}  (with $v=u_\rho$), and for $\tilde C$ enough large (it is enough to choose $\tilde C=\sqrt{13M}$), the point~$(\bar x, \bar y)$ belongs to~$\mathcal{B}$ and consequently, the function $\tilde{\psi}$ attains a (local) maximum at some point $(\tilde x, \tilde y)\in \mathcal{B}$, namely the function
$$
(x,y)\mapsto u^\eps(x,y)-\left[u_\rho(x)+\eps w(y;\bar x)+\eps\eta\chi(y)+\sigma|x|^2+c|x-\bar x|^2\right]
$$
attains a maximum at $(\tilde x, \tilde y)$.
Without loss of generality, we may assume that this maximum is strict and global. Indeed, otherwise, we can add to $u^\eps$ some smooth function vanishing with its first and second derivatives at $(\tilde x, \tilde y)$.
For any $\tau \in (0,1)$, we introduce the function
\begin{equation}\label{Psi}
\Psi(x, y, \xi)= u^\eps(x,y)-u_\rho(\xi)-\left[\eps w(y; \bar x)+\eps\eta\chi(y)+\sigma|x|^2+c|x-\bar x|^2+\tau|x-\xi|^2\right].
\end{equation}
By standard theory, we infer that the function $\Psi$ attains a maximum at some point $(x_\tau, y_\tau, \xi_\tau)$ with $x_\tau, \xi_\tau \to \tilde x$ and $y_\tau \to \tilde y$ as $\tau \to \infty$. 
We apply \cite[Theorem $3.2$]{CIL92} with $\mathcal{O}_1=\R^n \times \R^m$, $\mathcal{O}_2=\R^n$, $u_1=u^\eps$, $u_2=-u_\rho$ and
\[
\phi(x, y, \xi)=\eps w(y; \bar x)+\eps \eta \chi(y)+\sigma |x|^2+c|x-\bar x|^2+\tau|x-\xi|^2;
\]
for every $\bar \eps \in (0,1)$, there exist two matrices $X_1=\begin{pmatrix} X_{xx} & X_{xy} \\ X_{yx} & X_{yy} \end{pmatrix} \in \mathbb{M}^{n+m}$ and $X_2 \in \mathbb{M}^n$ such that
\begin{eqnarray}
&& \left(\left(2\sigma x_\tau+2c(x_\tau-\bar x)+2\tau(x_\tau-\xi_\tau), \eps D_y w(y_\tau; \bar x)+\eps \eta D_y \chi(y_\tau)\right), X_1\right) \in  J^+_{(x_\tau, y_\tau)} u^\eps\label{I1}\\
&& \left(2\tau(x_\tau-\xi_\tau), X_2\right)\in J^-_{\xi_\tau} u_\rho\label{II1}\\
&&\begin{pmatrix}X_1 & 0 \\ 0 &-X_2\end{pmatrix} \leq A +\bar \eps A^2, \textrm{ with $A=D^2_{x, y, \xi}\phi(x_\tau, y_\tau, \xi_\tau)$}\label{III}
\end{eqnarray}
where $J^+_{(x_\tau, y_\tau)} u^\eps$, $J^-_{\xi_\tau} u_\rho$ and $D^2_{x, y, \xi}\phi$ denote respectively the superjet of $u^\eps$ at $(x_\tau, y_\tau)$, the subjet of $u_\rho$ at $\xi_\tau$ and the Hessian of $\phi$ with respect to the three variables $(x, y, \xi)$.
Moreover, testing~\eqref{III} with vectors of the form $(v,0,v)$ and $(0,w,0)$, we get respectively
\begin{equation}\label{III1e2}
X_{xx}-X_2\leq 2(\sigma +c)I+\bar \eps ||A^2||I\qquad\textrm{and}\qquad
X_{yy}\leq \eps \eta D^2_{yy}\chi(y_\tau)+\eps D^2_{yy}w(y_\tau; \bar x) + \bar \eps ||A^2||I.
\end{equation}
Since $u^\eps$ is a viscosity subsolution of equation \eqref{eq:HJ-eps3}, by~\eqref{I1} we infer
\begin{multline*}
0\geq u^\eps(x_\tau, y_\tau)+H\left(x_\tau, 2\sigma x_\tau +2c(x_\tau-\bar x)+2\tau(x_\tau-\xi_\tau), X_{xx}\right)\\+\frac{1}{\eps}\mathcal{L}\left(y_\tau, \eps D_y w(y_\tau; \bar x)+\eps \eta D_y \chi(y_\tau), X_{yy}\right)+f(x_\tau, y_\tau).
\end{multline*}
Note $\mathcal{L}$ is uniformly elliptic by definition and $H$ is degenerate elliptic by assumption $(H)$. Then, by \eqref{III1e2}, we get
\begin{multline*}
0\geq u^\eps(x_\tau, y_\tau)+H\left(x_\tau, 2\sigma x_\tau +2c(x_\tau-\bar x)+2\tau(x_\tau-\xi_\tau), X_2+2(\sigma +c)I+\bar \eps ||A^2||I\right)\\+\frac{1}{\eps}\mathcal{L}\left(y_\tau, \eps D_y w(y_\tau; \bar x)+\eps \eta D_y \chi(y_\tau), \eps D_{yy }^2 w(y_\tau; \bar x)+\eps \eta D_{yy}^2\chi(y_\tau) +\bar \eps ||A^2||I\right)+f(x_\tau, y_\tau),
\end{multline*}
which, by linearity of $\mathcal{L}$, implies 
\begin{multline}\label{eq:Heps2}
0\geq u^\eps(x_\tau, y_\tau)+H\left(x_\tau, 2\sigma x_\tau +2c(x_\tau-\bar x)+2\tau(x_\tau-\xi_\tau), X_2+2(\sigma +c)I+\bar \eps ||A^2||I\right)\\+\mathcal{L}\left(y_\tau,  D_y w(y_\tau, \bar x), D_{yy}^2 w(y_\tau; \bar x)\right)+\eta \mathcal{L}\left(y_\tau, D_y \chi(y_\tau), D^2_{yy} \chi(y_\tau)\right)+\frac{\bar \eps}{\eps} ||A^2||\mathcal{L}(y_\tau, 0, I)+f(x_\tau, y_\tau).
\end{multline}
By \eqref{eq:cellpb3}, we have
\begin{eqnarray*}
\mathcal{L}\left(y_\tau, D_yw(y_\tau;\bar x), D_{yy}^2 w(y_\tau;\bar x)\right)+f(x_\tau,y_\tau)&=&\lambda(\bar x)+\left[f(x_\tau,y_\tau)- f(\bar x, y_\tau)\right]\\&=& \lambda(x_\tau)+\left[\lambda(\bar x)-\lambda(x_\tau)\right]+\left[f(x_\tau,y_\tau)- f(\bar x, y_\tau)\right].
\end{eqnarray*}
By assumption~$(H4)$ and by Corollary~\ref{corollary:A}, we have
\[
f(x_\tau,y_\tau)- f(\bar x,y_\tau)\geq -C_3|x_\tau-\bar x|^\beta\qquad \textrm{and}\qquad
\lambda(\bar x)-\lambda(x_\tau)
\geq -C_3|x_\tau-\bar x|^\beta.
\]
Replacing these relations into \eqref{eq:Heps2}, we have
\begin{multline*}
0\geq u^\eps(x_\tau, y_\tau)+H\left(x_\tau, 2\sigma x_\tau +2c(x_\tau-\bar x)+2\tau(x_\tau-\xi_\tau), X_2+2(\sigma +c)I+\bar \eps ||A^2||I\right)\\+\lambda(x_\tau)+\eta \mathcal{L}\left(y_\tau, D_y \chi(y_\tau), D^2_{yy} \chi(y_\tau)\right)+\frac{\bar \eps}{\eps} ||A^2||\mathcal{L}(y_\tau, 0, I)-2C_3|x_\tau-\bar x|^\beta.
\end{multline*}
Using the linear growth of $H$ w.r.t. to $p$ and $X$ (see assumption $(H)$) and the boundedness of the matrix $\tau$ (see assumption~$(L)$), we obtain
\begin{multline*}
0\geq u^\eps(x_\tau, y_\tau)+H\left(x_\tau, 2\tau(x_\tau-\xi_\tau), X_2\right)+\lambda(x_\tau)+\eta \mathcal{L}\left(y_\tau, D_y \chi(y_\tau), D^2_{yy} \chi(y_\tau)\right)+\\-C\left[\sigma |x_\tau|+c|x_\tau-\bar x|+|x_\tau-\bar x|^\beta+\sigma +c+\bar \epsilon ||A^2||+\frac{\bar \eps}{\eps} ||A^2||\right],
\end{multline*}
for some constant $C$ independent of $\tau, \sigma, \eps, \bar \eps$ and $\rho$. Note that from now on, with some abuse of notation, we will denote by $C$ different constants independent of $\tau, \sigma,\eps, \bar \eps$ and $\rho$.

Since $u_\rho$ is a supersolution to \eqref{eq:effprobrho}, by~\eqref{II1}, we have
$
u_\rho(\xi_\tau)+\overline H_\rho\left(\xi_\tau, 2 \tau(x_\tau-\xi_\tau), X_2\right)\geq 0.
$
We recall that $\xi_\tau \to \tilde x$ and $x_\tau \to \tilde x$ as $\tau \to \infty$; hence, for $\tau$ sufficiently small, by the definition of~$\overline H_\rho$ in~\eqref{eq:effbarrhoH}, we have
\begin{equation}\label{eq:barHrhomaj}
\overline H_\rho(\xi_\tau, 2 \tau(x_\tau-\xi_\tau), X_2)\leq  H(x_\tau, 2 \tau(x_\tau-\xi_\tau), X_2)+\lambda(x_\tau).
\end{equation}
By the last three inequalities, we get
\begin{multline*}
0\geq u^\eps(x_\tau, y_\tau)-u_\rho(\xi_\tau)+\eta \mathcal{L}\left(y_\tau, D_y \chi(y_\tau), D^2_{yy} \chi(y_\tau)\right)+\\-C\left[\sigma |x_\tau|+c|x_\tau-\bar x|+|x_\tau-\bar x|^\beta+\sigma +c+\bar \epsilon ||A^2||+\frac{\bar \eps}{\eps} ||A^2||\right].
\end{multline*}
Passing to the limit as $\bar \eps \to 0^+$ and recalling  that $C$ does not depend on $\bar \eps$, we obtain
$$
0\geq u^\eps(x_\tau, y_\tau)-u_\rho(\xi_\tau)+\eta \mathcal{L}\left(y_\tau, D_y \chi(y_\tau), D^2_{yy} \chi(y_\tau)\right)-C\left[\sigma |x_\tau|+c|x_\tau-\bar x|+\sigma +c+|\bar x -x_\tau|^\beta\right].
$$
Now passing to the limit as $\tau \to \infty$ and recalling that $C$ is independent of $\tau$ and that $x_\tau \to \tilde x$, $y_\tau \to \tilde y$ and $\xi_\tau\to \tilde x$ as $\tau \to + \infty$, we get
$$
0\geq u^\eps(\tilde x, \tilde y)-u_\rho(\tilde x)+\eta \mathcal{L}\left(\tilde y, D_y \chi(\tilde y), D^2_{yy} \chi(\tilde y)\right)-C\left[\sigma |\tilde x|+c|\tilde x-\bar x|+\sigma +c+|\bar x -\tilde x|^\beta\right],
$$
and since $|\tilde x -\bar x|\leq r$, we obtain
\begin{equation}\label{eqn:eqtest2}
u^\eps(\tilde x, \tilde y)-u_\rho(\tilde x)\leq C\left[\sigma(|\tilde x|+1)+cr+c+r^\beta\right]-\eta \mathcal{L}\left(\tilde y, D_y\chi(\tilde y), D^2_{yy}\chi(\tilde y)\right).
\end{equation}
Note that, by \eqref{lyapounov}, there exists  $\tilde{L}>0$ such that 
\begin{equation}\label{eq:lyapproof}
-\eta \mathcal{L}\left(\tilde y, D_y \chi(\tilde y), D^2_{yy} \chi(\tilde y)\right)\leq -\eta \tilde{L}.
\end{equation}
Moreover for every $(x,y) \in \R^n\times \R^m$, there holds
$
\psi(x,y)\leq \psi(\bar x, \bar y)=\tilde \psi(\bar x, \bar y)\leq \tilde \psi(\tilde x, \tilde y).
$
Namely, by using the previous relation, \eqref{eqn:eqtest2}, \eqref{eq:lyapproof} and Lemma \ref{estchiholder}-(ii), we have
\begin{eqnarray}\label{eqn:eqtest3bis}
&&u^\eps(x,y)-u_\rho(x)-\eps w(y;x)-\eps\eta\chi(y)-\sigma|x|^2\nonumber\\\nonumber &&\qquad\leq u^\eps(\tilde x, \tilde y)-u_\rho(\tilde x)-\eps w(\tilde y;\bar x)-\eps\eta \chi(\tilde y)-\sigma |\tilde x|^2-c|\tilde x-\bar x|^2\\ &&\qquad\leq C\left[\sigma(|\tilde x|+1)+cr+c+r^\beta\right]-\eta \tilde{L}+\eps\left[C_5\log C_5-C_5\log \eta+\eta\right].
\end{eqnarray}
By \eqref{limsigmabis} we deduce
\begin{equation}\label{limsigmabisbis}
\lim_{\sigma \to 0}\sigma |\tilde x|\leq\lim_{\sigma \to 0}[\sigma |\bar x|+\sigma|\tilde x-\bar x|]=\lim_{\sigma \to 0}\sigma |\tilde x-\bar x|\leq r\lim_{\sigma \to 0}\sigma=0.
\end{equation}
Letting $\sigma \to 0$ in  \eqref{eqn:eqtest3bis}, by \eqref{limsigmabisbis}, and choosing $\eta=\eps$, we obtain
\begin{equation*}
u^\eps(x,y)-u_\rho(x)\leq\eps\left[w(y;x)+\eps\chi(y)\right]+C\left[cr+c+r^\beta+\eps+\eps\left|\log\eps\right|+\eps^2\right].
\end{equation*}
Since we must choose $c$ and $r$ so to fulfill \eqref{eqn:cr_holder}, we take
$$
c=r^\beta=\left[\bar C\eps\left(1+2\left|\log\eps\right|\right)\right]^{\frac{\beta}{2}}
$$
and we get
\begin{multline*}
u^\eps(x,y)-u_\rho(x)\leq\eps\left[w(y;x)+\eps\chi(y)\right]+C\left[\eps\left(1+\left|\log\eps\right|\right)\right]^{\frac{\beta}{2}}\\
\leq\eps\left(\max_{\mathcal{K}}|w|+\eps \max_{\mathcal{K}}|\chi|\right)+C\left[\eps(1+|\log\eps|)\right]^{\frac{\beta}{2}};
\end{multline*}
hence, we obtain the statement where $K$ depends on $C,  \max_{\mathcal{K}}|w|, \max_{\mathcal{K}}|\chi|$.
\end{proof}
\begin{remark}
In fact, we proved that there exists a positive constant $K$ such that
$$u^\eps(x,y)-u_\rho(x)\le K\left[[\eps(1+|\log\eps|)]^{\frac{\beta}{2}}+\eps\log(1+|y|^2)+\eps^2|y|^2\right]\qquad \forall (x,y)\in \R^n\times \R^m.
$$
\end{remark}

\section{Particular cases}\label{sec:partcases}
In this Section we assume the following hypotheses
\begin{itemize}
\item[(LF)] the source term has separated variables: $f(x,y)=h(x)g(y)$ with $g$ bounded,  and Lipschitz continuous;
\item[(HF)] $h$ is bounded, $\|h\|_\infty\leq  C_1$, and Lipschitz continuous with Lipschitz constant~$L_h$.
\end{itemize}
Note that assumption~$(H2)$ is no more required.
Under these hypotheses, we obtain the same rate of convergence as in Theorem~\ref{thm:rategc} with~$\beta=1$ but without requiring $(H2)$: see Theorem~\ref{thm:ratehlip} below.
Moreover, for $h \in C^2$, we get a better rate of convergence: see Theorem~\ref{thm:ratehreg} below.

As in Section~\ref{sec:holder}, we first tackle the cell problem~\eqref{eq:cellpb3} because we need the uniqueness of its solution~$(w,\lambda)=(w(y;x),\lambda(x))$, the sublinear growth of~$w(\cdot;x)$, some regularity of~$\lambda$ and a continuous dependence estimate of~$w$ in~$x$.
These properties are established in the following statements; note that the results of Section~\ref{sec:holder} do not apply because they rely on Theorem~\ref{holder} which requires assumption~$(H2)$.
\begin{prop}
Assume $(LF)$ and $(HF)$. For any $x \in \R^n$ fixed, there exists a unique $\lambda(x)$ such that \eqref{eq:cellpb3} has a unique solution. The unique solution of \eqref{eq:cellpb3} is given by
$$
w(y;x)=h(x)w_*(y),
$$
where $(w_*,\lambda_1) \in C^2(\R^m)\times\R$ is the unique solution to the cell problem
\begin{equation}\label{eq:cellwstar}
\mathcal{L}\left(y, D_yw_*, D^2_{yy}w_*\right)+g(y)=\lambda_1, \quad w_*(0)=0 \qquad \mbox{ in }\R^m.
\end{equation}
Moreover there holds
\begin{equation}\label{eq:lambda1spcases}
\lambda(x)=h(x)\lambda_1=h(x)\int_{\R^m} g(y)d\mu(y) \qquad \forall x \in \R^n,
\end{equation}
where $\mu$ is the unique invariant measure associated to $\mathcal{L}$.
\end{prop}
\begin{proof}
The proof follows the same arguments of Proposition~\ref{prop:nota}, replacing Theorem~\ref{holder} with \cite[Proposition $5.2$]{Ghilli18}.
\end{proof}
The following estimates are analogous to those of Lemma \ref{estchiholder}. We omit the proof since it is analogous to that of Lemma \ref{estchiholder}, just replacing Theorem~\ref{holder} with \cite[Proposition $5.2$]{Ghilli18}.
\begin{lemma}\label{estchi}
Assume $(LF)$ and $(HF)$. Let $w_*$ be the solution to \eqref{eq:cellwstar}. Then
\begin{itemize}
\item[i)] there exists a $C_5>0$ such that 
$$
|w_*(y)|\leq C_5\left(1+\log(y^2+1)\right)\qquad \forall y \in \R^m;
$$
\item[ii)] Let $C_1$ and $C_5$ be the constants defined respectively in $(HF)$ and in point i). Let $\chi$ be the Lyapounov function of \eqref{lyapounov}. Then there holds
$$
|h(x)w_*(y)|-\eta \chi(y)\leq  C_1C_5\log C_1C_5- C_1C_5\log\eta +\eta \qquad \forall (x,y) \in \R^n \times \R^m.
$$
\end{itemize}
\end{lemma}
\begin{remark}
By the same arguments of the proof of Lemma~\ref{lemma:ex_u}, we obtain that the effective problem~\eqref{eq:effprob} admits exactly one bounded solution~$u$ which moreover fulfills: $\|u\|_\infty\leq C+C_1||g||_\infty$.
\end{remark}

\subsection{Lipschitz source term with separated variables}\label{sec:sepvar}
We can now establish an estimate of the rate of convergence when the cost is Lipschitz continuous and has separated variables.
\begin{theorem}\label{thm:ratehlip}
Assume $(LF)$ and $(HF)$. Let $u^\eps$ and ~$u$ be the solution to equation \eqref{eq:HJ-eps3} and respectively to~\eqref{eq:effprob}. For $\eps$ sufficiently small and for  every compact $\mathcal K \subset \R^m$ there exists a constant $K\geq 0$ such that 
\begin{equation*}
\left|u^\eps(x,y)-u(x)\right|\le K\left(\eps|\log\eps|\right)^{\frac{1}{2}}\qquad \forall (x,y)\in \R^n\times \mathcal{K}.
\end{equation*}
\end{theorem}
As in subsection \ref{subsec:rateconv}, in order to prove Theorem~\ref{thm:ratehlip}, it is expedient to study the approximated effective problems~\eqref{eq:effprobrho}.  We first state the following Proposition whose proof is postponed after the proof of Theorem~\ref{thm:ratehlip}.
\begin{prop}\label{prop:ratehlip}
Assume $(LF)$ and $(HF)$. Let $u^\eps$ and $u_\rho$ be the solution to the equation \eqref{eq:HJ-eps3} and respectively to \eqref{eq:effprobrho}. Then, for $\eps$ sufficiently small and for every compact $\mathcal{K} \subset \R^m$, there exists $K\geq 0$ (independent of $\rho \in (0,1)$) such that:
$$
u^\eps(x,y)-u_\rho(x)\le K\left(\eps|\log\eps|\right)^{\frac{1}{2}}\qquad \forall (x,y)\in \R^n\times \mathcal{K}.
$$
\end{prop}
Now we first prove Theorem \ref{thm:ratehlip} and after Proposition \ref{prop:ratehlip}.
\begin{proof}[Proof of Theorem \ref{thm:ratehlip}]
The results of Lemma~\ref{lem:approxrho} still hold true. Hence, by Lemma \ref{lem:approxrho} and Proposition \ref{prop:ratehlip}, as $\rho \to 0^+$, we get
$$
u^\eps(x,y)-u(x)\le K\left(\eps|\log\eps|\right)^{\frac{1}{2}}\qquad \forall (x,y)\in \R^n\times \mathcal{K}.
$$
The other inequality is established in a similar manner as in the proof of Theorem~\ref{thm:rategc}.
\end{proof}

\begin{proof}[Proof of Proposition \ref{prop:ratehlip}]
The proof follows the same steps of the proof of Proposition \ref{prop:ratehgc}, with the main difference that in the place of $w(y;x)$ we have now $h(x)w_*(y)$, and that now we can rely on the Lipschitz regularity of $h$. We shall give some details for completeness of exposition.

We consider the function $\psi(x,y)$ defined in \eqref{eq:psiprop11} where in the place of $w(y;x)$ we have now $h(x) w_*(y)$. The function $\psi$ attains its global maximum at some point $(\bar x, \bar y)$ and the results of Lemma \ref{lem:est} (with $v=u_\rho$) hold.
We introduce the function $\tilde \psi$ defined in \eqref{eq:tildepsiprop11} where in the place of $w(y;x)$ we have now $h(x) w_*(y)$ and $c>0$ is to be fixed later. We proceed as in the proof of Proposition \ref{prop:ratehgc}.  We have \eqref{eq:comppsi2} and \textbf{Case a} and \textbf{Case b}. 

\textbf{Case a:} \quad $x \in \partial B(\bar x, r)$. Note that \eqref{eq:tildepsipsi} is equivalent to
\begin{equation*}
cr^2\geq \eps w_*(y)[h(x)-h(\bar x)], \quad \forall y \in B\left(0,\frac{\tilde C}{\sqrt{\eps\eta}}\right).
\end{equation*}
By using the Lipschitz regularity of $h$,  Lemma \ref{estchi} i),   that $ y \in B\left(0,\frac{\tilde C}{\sqrt{\eps\eta}}\right)$ and taking $\eps, \eta$ enough small, we get that it suffices to have
\begin{equation}\label{eqn:cr}
cr\geq \bar C\eps\left(1+ |\log(\eps \eta)|\right),
\end{equation}
where $\bar C>0$ depends on $L_h$ and $C_5$.
Therefore if \eqref{eqn:cr} holds, we have \eqref{eq:tildepsipsicasea}.

\textbf{Case b:} \quad $y  \in \partial B\left(0,\frac{\tilde C}{\sqrt{\eps\eta}}\right)$. To have \eqref{eq:tildepsipsizero}, it suffices to have
\begin{equation}\label{eq:tildeCpartcases}
\tilde C^2\geq 4M+\eps C_1 C_5\left(1+\log\left(1+\frac{\tilde C^2}{\eps \eta}\right)\right)+\sigma(2|\bar x|+1)
\end{equation}
and as in the proof of Proposition \ref{prop:ratehgc}, the previous inequality is always true  for $\tilde C$ large enough when $\sigma$ and $\eps$ tend to zero (and, as we will choose, $\eta=\eps$). 
Therefore by \eqref{eq:tildeCpartcases}, we have \eqref{eq:tildepsipsicaseb}.
In any case, we have \eqref{eq:tildepsipsitotal}. 
Hence the function $\tilde{\psi}$ attains a (local) maximum at some point $(\tilde{x}, \tilde{y})\in \mathcal{B}$.
Without loss of generality, we may assume that this maximum is strict and global. 

Again by replacing $w(y;x)$ with $h(x)w_*(y)$, for any $\tau \in (0,1)$, we introduce $\Psi$ as in~\eqref{Psi} 
and denote the maximum point of $\Psi$ by $(x_\tau, y_\tau, \xi_\tau)$ with $x_\tau, \xi_\tau \to \tilde x$ and $y_\tau \to \tilde y$ as $\tau \to \infty$.
As before, by \cite[Theorem $3.2$]{CIL92} 
for every $\bar \eps \in (0,1)$, there exist two matrices $X_1 \in \mathbb{M}^{n+m}, X_2 \in \mathbb{M}^n$ such that:
\eqref{I1} holds with $\eps D_y w(y_\tau;\bar x)$ replaced by $\eps h(\bar x) D_yw_*(y_\tau)$, \eqref{II1} holds and \eqref{III1e2} holds with $\eps D^2_{yy}w(y_\tau; \bar x)$ replaced by $\eps h(\bar x)D^2_{yy}w_*(y_\tau)$. 
Since $u^\eps$ is a subsolution of \eqref{eq:HJ-eps3}, we deduce
\begin{multline}\label{eqn:testeq}
0\geq u^\eps(x_\tau, y_\tau)+H\left(x_\tau, 2\sigma x_\tau +2c(x_\tau-\bar x)+2\tau(x_\tau-\xi_\tau), X_{xx}\right)\\+\frac{1}{\eps}\mathcal{L}\left(y_\tau, \eps h(\bar x) D_y w_*(y_\tau)+\eps \eta D_y \chi(y_\tau), X_{yy}\right)+h(x_\tau)g(y_\tau).
\end{multline}
By the same passages as in the proof of Proposition \ref{prop:ratehgc}, and using that \eqref{eq:cellwstar} implies
$$
h(\bar x)\mathcal{L}(y_\tau, D_y w_*(y_\tau), D^2_{yy}w_*(y_\tau))+h(\bar x)g(y_\tau)=h(\bar x)\lambda_1,
$$
we get
\begin{multline}\label{eq:C}
0\geq u^\eps(x_\tau, y_\tau)+H\left(x_\tau, 2\tau(x_\tau-\xi_\tau), X_2\right)+h(x_\tau)\lambda_1+\eta \mathcal{L}\left(y_\tau, D_y \chi(y_\tau), D^2_{yy} \chi(y_\tau)\right)+\\-C\left[\sigma |x_\tau|+c|x_\tau-\bar x|+\sigma +c+\bar \epsilon ||A^2||+\frac{\bar \eps}{\eps} ||A^2||\right]+\left[h(\bar x)-h(x_\tau)\right]\lambda_1+\left[h(x_\tau)-h(\bar x)\right]g(y_\tau),
\end{multline}
for some constant $C$ independent of $\tau, \sigma, \eps, \bar \eps, \rho$. From now on, with some abuse of notation, we  will use the same symbol~$C$ for updates of a constant $C$ independent of $\tau, \sigma, \eps, \bar \eps, \rho$.

Since $u_\rho$ is a supersolution to \eqref{eq:effprobrho} and $\xi_\tau \to \tilde x$ and $x_\tau \to \tilde x$ as $\tau \to 0^+$ and then, for $\tau$ sufficiently small, we have \eqref{eq:barHrhomaj} with $h(x_\tau)\lambda_1$ in the place of $\lambda(x_\tau)$, we deduce
\begin{multline*}
0\geq u^\eps(x_\tau, y_\tau)-u_\rho(\xi_\tau)+\eta \mathcal{L}\left(y_\tau, D_y \chi(y_\tau), D^2_{yy} \chi(y_\tau)\right)\\-C\left[\sigma |x_\tau|+c|x_\tau-\bar x|+\sigma +c+\bar \epsilon ||A^2||+\frac{\bar \eps}{\eps} ||A^2||\right]-\left|h(\bar x)-h(x_\tau)\right|\left(|\lambda_1|+|g(y_\tau)|\right).
\end{multline*}
The Lipschitz continuity of $h$, the boundedness of $g$ stated in $(LF)$ and $(HF)$ and \eqref{eq:lambda1spcases} yield
$$
|h(\bar x)-h(x_\tau)|\left(|\lambda_1|+|g(y_\tau)|\right)\leq  2L_h||g||_\infty|\bar x-x_\tau|.
$$
By the last two inequalities, we get
\begin{multline*}
0\geq u^\eps(x_\tau, y_\tau)-u_\rho(\xi_\tau)+\eta \mathcal{L}\left(y_\tau, D_y \chi(y_\tau), D^2_{yy} \chi(y_\tau)\right)\\-C\left[\sigma (|x_\tau|+1)+(c+1)|x_\tau-\bar x| +c+\bar \epsilon ||A^2||+\frac{\bar \eps}{\eps} ||A^2||\right].
\end{multline*}

The rest of the proof is very similar to  the proof of Proposition \ref{prop:ratehgc}. We pass to the limit in the previous inequality first as $\bar \eps \to 0^+$, then as $\tau \to 0^+$ (recalling  that $C$ does not depend on $\bar \eps$ and $\tau$), we use $x_\tau \to \tilde x$, $y_\tau \to \tilde y$ and $\xi_\tau\to \tilde x$ ad $\tau \to + \infty$,  $|\tilde x -\bar x|\leq r$ and \eqref{eq:lyapproof}. The estimate obtained is used in \eqref{eqn:eqtest3bis} with Lemma \ref{estchi} ii) and  letting $\sigma \to 0$ and choosing $\eta=\eps$, we obtain
\begin{equation*}
u^\eps(x,y)-u_\rho(x)\leq\eps\left[h(x)w_*(y)+\eps\chi(y)\right]+C\left[cr+c+r+\eps+\eps|\log\eps|+\eps^2\right].
\end{equation*}
Recalling \eqref{eqn:cr}, we take
$$
c=r=\left[\bar C \eps(1+2|\log\eps|)\right]^{\frac{1}{2}}
$$
and we get
\begin{eqnarray*}
u^\eps(x,y)-u_\rho(x)&\leq& \eps\left(h(x)w_*(y)+\eps\chi(y)\right)+C\left[\eps(1+|\log\eps|)\right]^{\frac{1}{2}}\\ &\leq& \eps\left(C_1\max_{\mathcal{K}}|w_*|+\eps \max_{\mathcal{K}}|\chi|\right)+C\left[\eps(1+|\log\eps|)\right]^{\frac{1}{2}}
\end{eqnarray*}
 and we conclude the statement for $K$ depending on $C, C_1, \max_{\mathcal{K}}|w_*|, \max_{\mathcal{K}}|\chi|$.
\end{proof}

\subsection{Smooth source term with separated variables}
In this subsection, besides assumption $(LF)$, we replace condition~$(HF)$ with the stronger condition
\begin{itemize}
\item[$(HC)$] $h \in C^2(\R^n)$ with $\|h\|_{C^2}:=\|h\|_\infty+\|D_xh\|_\infty+\|D^2_{xx}h\|_\infty \leq C_1$.
\end{itemize}
Under these hypotheses, in the following Theorem we obtain a better rate of convergence. The proof is similar and in part simpler than that of Theorem~\ref{thm:ratehlip}; we only write the proof of Proposition~\ref{prop:rchc2} which is the counterpart of Proposition~\ref{prop:ratehgc} in this case.
\begin{theorem}\label{thm:ratehreg}
Assume $(LF)$ and $(HC)$. Let $u^\eps$ and~$u$ be the solution to \eqref{eq:HJ-eps3} and respectively to~\eqref{eq:effprob}. Then, for $\eps$ sufficiently small and for every compact $\mathcal{K} \subset \R^m$,  there exists a constant $K$ such that 
$$
\left|u^\eps(x,y)-u(x)\right|\leq K\eps |\log(\eps)|\qquad \forall (x,y)\in \R^n\times \mathcal{K}.
$$
\end{theorem}

\begin{prop}\label{prop:rchc2}
Assume $(LF)$ and $(HC)$. Let $u^\eps$ and~$u_\rho$ be the solution to~\eqref{eq:HJ-eps3} and respectively to~\eqref{eq:effprobrho}. For $\eps$ sufficiently small and for every compact $\mathcal{K} \subset \R^m$,  there exists a constant $K$ such that 
$$
u^\eps(x,y)-u_\rho(x)\leq K\eps|\log(\eps)|\qquad \forall (x,y)\in \R^n\times \mathcal{K}.
$$
\end{prop}
\begin{proof}
The function $\psi$ defined in the proof of Proposition \ref{prop:ratehlip} has a maximum point at $(\bar x, \bar y)$. The proof follows the same steps as the proof of Proposition \ref{prop:ratehlip}, with the main difference that now, since $h$ is smooth in $x$, we do not need to localize $\psi$ around $\bar x$ (hence there is no need to define the function $\tilde \psi$) and the term $c|x-\bar x|^2$ is no more present. Now, for $\tau \in (0,1)$, we introduce the function 
$$
\Psi(x, y, \xi)= u^\eps(x,y)-u_\rho(\xi)-\left[\eps h(x)w_*(y)+\eps\eta\chi(y)+\sigma|x|^2+\tau|x-\xi|^2\right]
$$
which has a maximum point in~$(x_\tau, y_\tau, \xi_\tau)$ with $x_\tau, \xi_\tau \to \bar x$ and $y_\tau \to \bar y$ as $\tau \to \infty$.
We proceed as in the proof of Proposition \ref{prop:ratehlip}. By \cite[Theorem $3.2$]{CIL92} we get that, for every $\bar \eps \in (0,1)$, there exist two matrices $X_1 \in \mathbb{M}^{n+m}, X_2 \in \mathbb{M}^n$ such that relation~\eqref{II1} holds while relations~\eqref{I1} and \eqref{III1e2} become respectively
\begin{equation*}
\begin{array}{l}
\left(\left(2\sigma x_\tau+\eps D_xh(x_\tau)w_*(y_\tau)+2\tau(x_\tau-\xi_\tau), \eps h(x_\tau) D_y w_*(y_\tau)+\eps \eta D_y \chi(y_\tau)\right), X_1\right) \in  J^+_{(x_\tau, y_\tau)} u^\eps\\
X_{xx}-X_2\leq \eps D^2_{xx}h(x_\tau) w_*(y_\tau)+2\sigma I+\bar \eps ||A^2||I\\
X_{yy}\leq \eps \eta D^2_{yy}\chi(y_\tau)+\eps h(x_\tau)D^2_{yy}w_*(y_\tau) + \bar \eps ||A^2||I.
\end{array}
\end{equation*}
Since $u^\eps$ is a subsolution of equation \eqref{eq:HJ-eps3}, we infer
\begin{multline*}
0\geq u^\eps(x_\tau, y_\tau)+H\left(x_\tau, 2\sigma x_\tau +\eps D_xh(x_\tau)w_*(y_\tau)+2\tau(x_\tau-\xi_\tau), X_{xx}\right)\\+\frac{1}{\eps}\mathcal{L}\left(y_\tau, \eps h(x_\tau) D_y w_*(y_\tau)+\eps \eta D_y \chi(y_\tau), X_{yy}\right)+h(x_\tau)g(y_\tau).
\end{multline*}
Similarly to the proof of Proposition \ref{prop:ratehgc} and since \eqref{eq:cellwstar} implies
$$
h(x_\tau)\mathcal{L}\left(y_\tau, D_y w_*(y_\tau), D^2_{yy}w_*(y_\tau)\right)+h(x_\tau)g(y_\tau)=h(x_\tau)\lambda_1,
$$
by using $(HC)$, we get
\begin{multline*}
0\geq u^\eps(x_\tau, y_\tau)+H\left(x_\tau, 2\tau(x_\tau-\xi_\tau), X_2\right)+h(x_\tau)\lambda_1+\eta \mathcal{L}\left(y_\tau, D_y \chi(y_\tau), D^2_{yy} \chi(y_\tau)\right)\\-CC_1\eps|w_*(y_\tau)|-C\left[\sigma |x_\tau|+\sigma +\bar \epsilon ||A^2||+\frac{\bar \eps}{\eps} ||A^2||\right],
\end{multline*}
for some constant $C$ independent of $\tau, \sigma, \eps, \bar \eps, \rho$. In the following, by some abuse of notation, we will use the same symbol $C$ for different constants independent of $\tau, \sigma, \eps, \bar \eps, \rho$.

The rest of the proof is the same as  in the proof of Proposition \ref{prop:ratehlip}. We pass to the limit first as $\bar \eps \to 0^+$, then as $\tau \to 0^+$ (recall that $C$ is independent of $\bar \eps$ and $\tau$), we use  $x_\tau \to \bar x$, $y_\tau \to \bar y$ and $\xi_\tau\to \bar x$ as $\tau \to\infty$. Recall that $(\bar x, \bar y)$ is a global maximum point for $\Psi(x,y)$. For $\eta=\eps$, using \eqref{eq:lyapproof}, taking the limit $\sigma \to 0$, (and again  $\sigma |\bar x|\to 0$) and using Lemma \ref{estchi} ii), we get
$$
u^\eps(x,y)-u_\rho(x)\leq \eps(h(x)w_*(y)+\eps \chi(y))+\eps \left(C_1C_5\left(\log C_1C_5-\log \eps\right)+\eps\right)+CC_1\eps|w_*(\bar y)|-\eps \tilde{L}.
$$
Moreover, since for $\eps$ enough small,  by Lemma \ref{lem:est} ii) (with $v=u_\rho$), $\bar y$ belongs to some ball (whose radius depends only on $\eps$ and $\eta$, but not on $\sigma$ and $\rho$), we may assume that, as $\sigma\to0^+$, $\bar y$ converge to some $\tilde{y}$ with
\begin{equation}\label{9bis}
|\tilde y|\leq \frac{\sqrt{13M}}{\eps}.
\end{equation}
Therefore, by Lemma \ref{estchi} i) and \eqref{9bis}, we have
$$
|w_*(\tilde y)|\leq C_5\left(1+\log(1+|\tilde y|^2)\right)\leq C_5\left(1+\log\left(1+\frac{13M}{\eps^2}\right)\right)\leq C-2\log\eps.
$$
By the last two inequalities we have for every compact $\mathcal{K} \subset \R^m$
\begin{eqnarray*}
u^\eps(x,y)-u_\rho(x)&\leq& \eps(h(x)w_*(y)+\eps\chi(y))+\eps \left(C_1C_5\left(\log C_1C_5+|\log \eps|\right)+\eps\right)+C C_1\eps\left(C+2|\log\eps|\right)
\\&\leq& \eps \left(C_1\max_{\mathcal{K}}|w_*|+\eps \max_{\mathcal{K}}|\chi|\right)+CC_1\eps (1+|\log \eps|),
\end{eqnarray*}
 and we conclude the statement for $K$ depending on $C, C_1, \max_{\mathcal{K}}|w_*|, \max_{\mathcal{K}}|\chi|$.
\end{proof}

\begin{appendices}
\section{Proof of Theorem \ref{holder}}
\begin{proof}
(i). This claim follows by standard theory (see \cite[Theorem II.1]{AL98}; see  also \cite[Lemma $3.5$]{MMT18} for a similar argument). However, for the sake of completeness, we provide the main features. We note that $u^\pm:= \pm \|F\|_\infty/\delta$ are respectively a super- and a subsolution to problem~\eqref{delta_ergo}. Hence, applying Perron's method, we infer the existence of a solution $u_\delta$ with $u^-\leq u_\delta\leq u^+$. Note that since $\|F\|_\infty\leq K_F$ we have that $\|u_\delta\|_\infty \leq \frac{K_F}{\delta}$, namely for each $\delta \in (0,1)$ the solution $u^\delta$ is bounded.

(ii). Adapting the arguments of \cite[Theorem $4.3$]{FIL}, we introduce:
$$
w_\delta(x,y):=u_\delta(x)-u_\delta(y)
$$
$$
g(x,y):=K_1\left[d+\left(d^2+|x-y|^\gamma\right)\left(\log(1+|x|^2)+\log(1+|y|^2)+K_2\right)\right].
$$
We want to prove the following statement. 
There exist $K_1, K_2>0$, both independent of $\delta$, such that for all $\delta \in (0,1)$, for $d$ sufficiently small (depending on $\delta$), there holds 
$$
w_\delta(x,y) \leq g(x,y) \quad \forall (x,y) \in \R^m\times \R^m.
$$
Assume for the moment that the claim is true. Hence for any $\delta>0$, for $d$ sufficiently small, we have
$$
w_\delta(x,y)\leq K_1\left[d+\left(d^2+|x-y|^\gamma\right)\left(\log(1+|x|^2)+\log(1+|y|^2)+K_2\right)\right]
$$
and letting $d \to 0^+$ we obtain
$$
u_\delta(x)-u_\delta(y)\leq K_1 |x-y|^\gamma\left(\log(1+|x|^2)+\log(1+|y|^2)+K_2\right)
$$
which is equivalent to the statement by arbitrariness of $x$ and $y$.\\
It remains to prove the claim. We argue by contradiction assuming
$$
\sup_{\R^m \times \R^m}(w_\delta-g)>0.
$$
Since $u_\delta$ is bounded and $\lim_{(x,y)\to \infty} g=\infty$, we
have $w_\delta \leq g$ in $\R^m \times \R^m \setminus B_R$ for a suitable ball $B_R \subset \R^m \times \R^m$. In particular $w_\delta-g$ admits a maximum point that we denote by $(\hat x, \hat y)$. We have
\[
w_\delta(\hat x, \hat y)-g(\hat x, \hat y)>0\qquad\textrm{and}\qquad \hat x \neq \hat y.
\]
We introduce the matrix $\Sigma(x,y)$ and the operator $\Xi$:
$$
\Sigma(x,y)=\begin{pmatrix} \tau(x) \tau(x)^T &\tau(x)\tau(y)^T\\
\tau(y)\tau(x)^T &\tau(y)\tau(y)^T
 \end{pmatrix}, \quad 
\Xi v(x,y)=\mbox{tr}\left[\Sigma(x,y)D_{xy}^2v\right],
$$
where $D_{xy}^2 v$ denotes the Hessian of $v$ with respect both to $x$ and $y$. 
Note that $\Sigma$ is semidefinite positive for all $(x,y) \in \R^m \times \R^m$.

\textbf{Claim 1}:  there holds
\begin{multline*}
\delta w_\delta(\hat x, \hat y)-\Xi w_\delta(\hat x, \hat y) +\alpha \hat x \cdot D_xw_\delta(\hat x, \hat y) +\alpha \hat y \cdot D_y w_\delta(\hat x, \hat y)+b(\hat x)\cdot D_x w_\delta(\hat x, \hat y) +b(\hat y)\cdot D_y w_\delta(\hat x, \hat y)\\ \leq C_F|\hat x -\hat y|^\gamma \left[\log\left(1+|\hat x|^2\right)+\log\left(1+|\hat y|^2\right)+1\right].
\end{multline*}
Indeed we observe that
\begin{eqnarray*}
\Xi w_\delta=\mbox{tr}\left(\tau(x)\tau(x)^TD^2_{xx}u_\delta(x)\right)-\mbox{tr}\left(\tau(y)\tau(y)^TD^2_{yy}u_\delta(y)\right).
\end{eqnarray*} 
Then we deduce
\begin{multline*}
\delta w_\delta(\hat x, \hat y)-\Xi w_\delta(\hat x, \hat y)+(\alpha \hat x +b(\hat x))\cdot D_xw_\delta(\hat x, \hat y)+(\alpha \hat y +b(\hat y))\cdot D_y w_\delta(\hat x, \hat y)\\ =F(\hat x)-F(\hat y)\leq C_F|\hat x-\hat y|^\gamma\left[\log\left(1+|\hat x|^2\right)+\log\left(1+|\hat y|^2\right)+1\right],
\end{multline*}
where the last inequality is due to $(F2)$.

\textbf{Claim 2}: there holds
\begin{multline*}
\delta g(\hat x,\hat y)-\Xi g(\hat x, \hat y)+\left(\alpha \hat x+b(\hat x)\right)D_x g(\hat x, \hat y)+\left(\alpha \hat y+b(\hat y)\right)D_y g(\hat x, \hat y)\\ \leq C_F|\hat x-\hat y|^\gamma \left[\log\left(1+|\hat x|^2\right)+\log\left(1+|\hat y|^2\right)+1\right].
\end{multline*}
Indeed it suffices to recall that $(\hat x, \hat y)$ is a maximum point for $w_\delta -g$, that \textbf{Claim 1} holds and to apply the maximum principle (using that $\Sigma$ is semidefinite positive).

\textbf{Claim 3}: there exist $K_1$ and $K_2$ such that: for any $\delta \in (0,1)$, there exists $d_0>0$ (dependent on $\delta$) such that the function $g$ verifies
\begin{multline*}
\delta g(\hat x, \hat y)-\Xi g(\hat x, \hat y)+(\alpha \hat x+b(\hat x))D_x g(\hat x, \hat y)+(\alpha \hat y+b(\hat y))D_y g(\hat x, \hat y)\\>C_F |\hat x-\hat y|^\gamma\left[\log\left(1+|\hat x|^2\right)+\log\left(1+|\hat y|^2\right)+1\right] \quad \forall d \in (0,d_0).
\end{multline*}
Note that \textbf{Claim 3} contradicts \textbf{Claim 2} so accomplishes the proof. Hence, we are left with proving \textbf{Claim 3}. In order to do this, we explicitly calculate the lefthand side of \textbf{Claim~2}. We need the following calculations.
For functions $T\,:\, \R^m\to \R,\, \psi:\R^m\to \R$, we have
$$
\Xi \left[T(x-y)\left(\psi(x)+\psi(y)\right)\right]=\mbox{tr}(C_1+C_2+C_3+C_4),
$$
where
\begin{eqnarray*}
C_1&=&\left[\tau(x)-\tau(y)\right]\left[\tau(x)-\tau(y)\right]^TD^2T(x-y)(\psi(x)+\psi(y))\\
C_2&=&\tau(x)\left[\tau(x)-\tau(y)\right]^T DT \otimes D\psi(x)+\left[\tau(x)-\tau(y)\right]\tau(x)^TDT \otimes D\psi(x)\\&&+\left[\tau(x)-\tau(y)\right]\tau(y)^TDT \otimes D\psi(y)+\tau(y)\left[\tau(x)-\tau(y)\right]^TDT \otimes D \psi(y)\\
C_3&=&\tau(x)\tau(x)^TD^2\psi(x), \quad C_4=\tau(y)\tau(y)^TD^2\psi(y).
\end{eqnarray*}
We choose
$$
T(x-y)=K_1(d^2+|x-y|^\gamma); \quad \psi(z)=\log\left(1+|z|^2\right)+\frac{K_2}{2}.
$$
Note that $g(x,y)=T(x-y)(\psi(x)+\psi(y))+K_1 d$.
Then we deduce
$$
-\Xi g=-\mbox{tr}(C_1)-\mbox{tr}(C_2)-\mbox{tr}(C_3)-\mbox{tr}(C_4).
$$
From now on, $C$ is a constant independent of $K_1, K_2, d, \delta$ (and may depends on $m, \tau$ and $b$), which may change from line to line. We have
\begin{eqnarray*}
-\mbox{tr}(C_1)&\geq& -L_\tau^2|x-y|^2\left[K_1 \gamma(2-\gamma)|x-y|^{\gamma-2}+m K_1|x-y|^{\gamma-2}\right](\psi(x)+\psi(y))\\ &\geq &-L_\tau^2K_1\gamma (m+2-\gamma)|x-y|^\gamma\left[\log\left(1+|x|^2\right)+\log\left(1+|y|^2\right)+K_2\right]\\
-\mbox{tr}(C_2)&\geq& -||\tau||_\infty L_\tau |x-y|K_1\gamma |x-y|^{\gamma-1}C\left[\frac{|x|}{1+|x|^2}+\frac{|y|}{1+|y|^2}\right]\\
-\mbox{tr}(C_3)&\geq& -||\tau||_\infty K_1\left(d^2+|x-y|^\gamma\right)\frac{C}{1+|x|^2}\\
-\mbox{tr}(C_4)&\geq& -||\tau||_\infty K_1\left(d^2+|x-y|^\gamma\right)\frac{C}{1+|y|^2}.
\end{eqnarray*}
Hence we get
\begin{multline*}
-\Xi g(\hat x, \hat y)\geq -K_1L_\tau^2\gamma(m+2-\gamma)|\hat x-\hat y|^\gamma \left[\log\left(1+|\hat x|^2\right)+\log\left(1+|\hat y|^2\right)+K_2\right]\\-K_1 C|\hat x-\hat y|^\gamma-2K_1d^2||\tau||_\infty^2.
\end{multline*}
Moreover, by definition of $g$, we have
$$
 (\alpha \hat x +b(\hat x))D_x g(\hat x, \hat y)+(\alpha \hat y +b(\hat y))D_y g(\hat x,\hat y)=D_1+D_2+ D_3
$$
where
\begin{eqnarray*}
D_1&=&\alpha (\hat x-\hat y)DT(\hat x-\hat y)(\psi(\hat x)+\psi(\hat y))+\alpha T(\hat x -\hat y)\left[\hat x \cdot D \psi(\hat x)+y \cdot D\psi(\hat y)\right]\\ &\geq &\alpha(\hat x-\hat y)K_1 \gamma |\hat x -\hat y|^{\gamma-2}(\hat x-\hat y)\left[\log\left(1+|\hat x|^2\right)+\log\left(1+|\hat y|^2\right)+K_2\right]\\ &&+\alpha K_1\left(d^2+|\hat x-\hat y|^\gamma\right)\left[\frac{2|\hat x|^2}{1+|\hat x|^2}+\frac{2 |\hat y|^2}{1+|\hat y|^2}\right]\\
D_2&=&\left[b(\hat x)-b(\hat y)\right]DT(\hat x-\hat y)(\psi(\hat x)+\psi(\hat y))\\ &\geq &-L_b|\hat x-\hat y|K_1 \gamma |\hat x -\hat y|^{\gamma-1}\left(\log\left(1+|\hat x|^2\right)+\log\left(1+|\hat y|^2\right)+K_2\right)\\
D_3&=&b(\hat x)T(\hat x-\hat y)D\psi(\hat x)+b(\hat y)T(\hat x-\hat y)D\psi(\hat y)\geq-CK_1\left(d^2+|\hat x-\hat y|^\gamma\right).
\end{eqnarray*}
The desired inequality of \textbf{Claim 3} is 
ensured by
\begin{equation}\label{eq:E123}
E_1+E_2+E_3>C_F|\hat x-\hat y|^\gamma\left[\log\left(1+|\hat x |^2\right)+\log\left(1+|\hat y|^2\right)+1\right],
\end{equation}  
where 
\begin{multline*}
E_1=K_1\delta d+K_1\delta d^2 K_2+K_1 \delta d^2\left[\log\left(1+|\hat x|^2\right)+\log\left(1+|\hat y|^2\right)+K_2\right]\\+K_1\delta|\hat x-\hat y|^\gamma\left[\log\left(1+|\hat x|^2\right)+\log\left(1+|\hat y|^2\right)+K_2\right]
\end{multline*}
$$
E_2=-2K_1d^2||\tau||_\infty^2-K_1L_\tau^2\gamma(m+2-\gamma)|\hat x-\hat y|^\gamma \left[\log\left(1+|\hat x|^2\right)+\log\left(1+|\hat y|^2\right)+K_2\right]-K_1 C|\hat x-\hat y|^\gamma
$$
\begin{multline*}
  E_3=K_1\alpha\gamma |\hat x-\hat y|^\gamma \left[\log\left(1+|\hat x|^2\right)+\log\left(1+|\hat y|^2\right)+K_2\right]-K_1 C d^2-K_1C|\hat x-\hat y|^\gamma\\-K_1 L_b \gamma |\hat x-\hat y|^\gamma \left[\log\left(1+|\hat x|^2\right)+\log\left(1+|\hat y|^2\right)+K_2\right]+\alpha K_1\left(d^2+|\hat x-\hat y|^\gamma \right)\left[\frac{2|\hat x|^2}{1+|\hat x|^2}+\frac{2|\hat y|^2}{1+|\hat y|^2}\right].
\end{multline*}
Inequality~\eqref{eq:E123} is ensured by
\begin{multline*}
|\hat x-\hat y|^\gamma\left[\log\left(1+|\hat x|^2\right)+\log\left(1+|\hat y|^2\right)+1\right]\left[K_1 \delta -K_1 L_\tau^2\gamma(m+2-\gamma)+K_1\alpha \gamma-K_1\gamma L_b-C_F\right]\\ +K_1\left[\delta d-Cd^2\right]+|\hat x -\hat y|^\gamma \left[(K_2-1)K_1\left(\delta -L_\tau^2\gamma(m+2-\gamma)+\alpha \gamma -\gamma L_b\right)-K_1C\right]>0.
\end{multline*}
We observe that there holds
\begin{equation}\label{eq:321}
|\hat x-\hat y|^\gamma\left[\log\left(1+|\hat x|^2\right)+\log\left(1+|\hat y|^2\right)+1\right]\left[K_1 \delta -K_1 L_\tau^2\gamma(m+2-\gamma)+K_1\alpha \gamma-K_1\gamma L_b-C_F\right]>0
\end{equation}
provided that $K_1\left[\delta -L_\tau^2\gamma(m+2-\gamma)+\alpha \gamma -\gamma L_b\right]>C_F$; in turns, this inequality is ensured by assumption~$(H2)$ and choosing $K_1$ sufficiently large.\\
Moreover there holds
\begin{equation}\label{eq:blue}
|\hat x -\hat y|^\gamma \left[(K_2-1)K_1\left(\delta -L_\tau^2\gamma(m+2-\gamma)+\alpha \gamma -\gamma L_b\right)-K_1C\right]>0
\end{equation}
provided that $\delta-L_\tau^2\gamma(m+2-\gamma)+\alpha \gamma -\gamma L_b>\frac{C}{K_2-1}$; in turns, this inequality is ensured by $(H2)$ and choosing $K_2$ sufficiently large.
Finally,
$K_1\left(\delta d-Cd^2\right)>0$
is ensured by $d<\frac{C}{\delta}$.

In conclusion (recall that $C$ depends only on $m, \tau $ and $b$), by~$(H2)$, we consider: (i) $d_0=\frac{C}{2\delta}$, (ii) $K_1$ sufficiently large to have \eqref{eq:321}, (iii) $K_2$ sufficiently large to have \eqref{eq:blue}.
Hence, we achieve the statement of \textbf{claim 3}. This concludes the proof.
\end{proof}

\section{Preliminary estimates}

In the following lemma we establish some useful estimates for proving Proposition~\ref{prop:ratehgc}, Proposition~\ref{prop:ratehlip} and Proposition~\ref{prop:rchc2}.
Consider a family of bounded functions $u^\eps=u^\eps(x,y)$ and a bounded function $v=v(x)$. Set $M=\sup\left\{||u^\eps||_\infty, ||v||_\infty\right\}$. Let $w=w(y;x)$ be a function which satisfies the two estimates in Lemma \ref{estchiholder} i) and ii). The same result holds if $w$ satisfies the two estimates of Lemma \ref{estchi} i) and ii). Since the proof is the same we give it in the case of Lemma \ref{estchiholder}.
For $\eps, \eta, \sigma >0$ and $\chi(y)=|y|^2$, we consider the function
$$
\psi(x,y):=u^\eps(x,y)-v(x)-\eps w(x;y)-\eps \eta\chi(y)-\sigma |x|^2.
$$
By the boundedness of $u^\eps$ and $v$, by  the at most logarithmic growth at infinity of $w$, the function $\psi$ must have a global maximum point $(\bar x, \bar y)$.  
\begin{lemma}\label{lem:est}
 The following estimates hold for $\eps$ small enough:
\begin{itemize}
\item[i)] $|\bar x|\leq \frac{\sqrt{6M-\eps C_5\log \eta}}{\sqrt{\sigma}}$
\item[ii)] $|\bar y|\leq \sqrt{\frac{12M-2\eps C_5\log\eta}{\eps\eta}}$
\item[iii)] $\sigma |\bar x|^2\leq 5M$, for $\eta=\eps$. 
\end{itemize} 
\end{lemma}
\begin{proof}
First we prove i). By the estimate in Lemma \ref{estchiholder} ii) and for $\eps$ small enough, we have
\begin{eqnarray*}
\psi(x,y)\leq  2M+\eps\left(C_5\log C_5 -C_5 \log \eta+ \eta\right)-\sigma|x|^2 \leq 4M-\eps C_5\log \eta-\sigma |x|^2,
\end{eqnarray*}
hence
$
\psi(\bar x,\bar y)\leq 4M-\eps C_5\log \eta-\sigma |\bar x|^2\leq -2M < \psi(0,0),
$
provided that 
$
\sigma |\bar x|^2> 6M -\eps C_5\log \eta.
$
Then, since $(\bar x, \bar y)$  is a global maximum point, we deduce i).

Now we prove ii). By Lemma \ref{estchiholder} ii), for $\eps$ small enough, we have
\[
\psi(\bar x, \bar y)
\leq 2M-\eps w(\bar y, \bar x)-\frac{\eps \eta}{2}|\bar y|^2-\frac{\eps \eta}{2}|\bar y|^2-\sigma |\bar x|^2\leq 4M-\eps C_5\log \eta -\frac{\eps \eta}{2}|\bar y|^2-\sigma |\bar x|^2.
\]
Hence we have
$
\psi(\bar x, \bar y)\leq 4M-\eps C_5 \log \eta -\frac{\eps \eta}{2}|\bar y|^2-\sigma |\bar x |^2\leq -2M-\sigma |\bar x|^2< \psi(\bar x, 0)
$
provided that 
$
\frac{\eps \eta}{2}|\bar y|^2> 6M-\eps C_5 \log \eta.
$

Now we prove iii). Since $\psi(\bar x, \bar y)\geq \psi(0, \bar y)$, we have
$$
u^\eps(\bar x, \bar y)-v(\bar x)-w(\bar y; \bar x)-\eps \eta \chi(\bar y)-\sigma|\bar x|^2\geq u^\eps(0,\bar y)-v(0)-\eps w(\bar y;0)-\eps \eta \chi(\bar y).
$$
By Lemma \ref{estchiholder} i) and taking $\eta=\eps$, we get
\begin{eqnarray*}
\sigma |\bar x|^2&\leq& 4M+\eps (w(\bar y, 0)-w(\bar y,\bar x))\leq 4M+2\eps C_5\left(1+\log\left(1+|\bar y|^2\right)\right)\\
&\leq &4M+2\eps C_5\left(1+\log\left(\frac{12M-2\eps C_5\log\eps+\eps^{2}}{\eps^{2}}\right)\right)
\end{eqnarray*}
and the proof is completed by noticing that 
\begin{eqnarray*}
\eps\log\left(\frac{12M-2\eps C_5\log\eps+\eps^{2}}{\eps^{2}}\right) =\eps\log\left(12M-2\eps C_5\log\eps+\eps^{2}\right)-2\eps\log\eps=o(1) \mbox{ as } \eps \to 0.
\end{eqnarray*}
\end{proof}
\end{appendices}

\subsection{Acknowledgements}
The authors want to thank prof. Paola Mannucci and prof. Nicoletta Tchou for many discussions and several useful suggestions. The second author was partially supported by GNAMPA-INDAM.


\end{document}